\documentclass[10pt,twoside,english]{amsart}
\normalsize
\advance\oddsidemargin by -1.0cm
\advance\evensidemargin by -1.0cm
\advance\topmargin by -1.0cm 
\usepackage{amssymb}
\usepackage{babel}
\usepackage{amsmath}
\usepackage{amscd}   
\usepackage{epsfig}
\usepackage{rotating}
\usepackage{arydshln}
\usepackage{float}
\usepackage{enumerate}
\let\mathg\mathfrak
\theoremstyle{plain}
\newtheorem{cor}{Corollary}[section]
\newtheorem{lem}[cor]{Lemma}
\newtheorem{thm}[cor]{Theorem}
\newtheorem{prop}[cor]{Proposition}

\theoremstyle{definition}
\newtheorem{exa}[cor]{Example}
\newtheorem{NB}[cor]{Remark}
\newtheorem{dfn}[cor]{Definition}

\numberwithin{equation}{section}
\newcommand{\bdm}{\begin{displaymath}}
\newcommand{\edm}{\end{displaymath}}
\newcommand{\be}{\begin{equation}}
\newcommand{\ee}{\end{equation}}
\newcommand{\ba}[1]{\begin{array}{#1}}
\newcommand{\ea}{\end{array}}
\newcommand{\bea}[1][]{\begin{eqnarray#1}}
\newcommand{\eea}[1][]{\end{eqnarray#1}}
\newcommand{\btab}{\begin{tabular}}
\newcommand{\etab}{\end{tabular}}

\usepackage{rotating}

\newcommand{\x}{\times}
\newcommand{\op}{\oplus}
\newcommand{\ox}{\otimes}

\newcommand{\ra}{\rightarrow}
\newcommand{\lra}{\longrightarrow}

\newcommand{\rk}{\ensuremath{\mathrm{rk}\,}}

\newcommand{\Id}{\ensuremath{\mathrm{Id}}}
\newcommand{\tr}{\ensuremath{\mathrm{tr}}}

\newcommand{\ad}{\ensuremath{\mathrm{ad}}}

\newcommand{\cyclic}[1]{\stackrel{{\scriptsize #1}}{\mathfrak{S}}}
\newcommand{\Hol}{\ensuremath{\mathrm{Hol}}}
\newcommand{\C}{\ensuremath{\mathbb{C}}}

\newcommand{\R}{\ensuremath{\mathbb{R}}}

\newcommand{\Q}{\ensuremath{\mathbb{Q}}}

\newcommand{\CP}{\ensuremath{\mathbb{CP}}}

\newcommand{\vphi}{\ensuremath{\varphi}}    
\newcommand{\vrho}{\ensuremath{\varrho}}

\newcommand{\kr}{\ensuremath{\mathcal{R}}}
\newcommand{\Ric}{\ensuremath{\mathrm{Ric}}}

\newcommand{\Lin}{\ensuremath{\mathrm{span}}}

\newcommand{\Ad}{\ensuremath{\mathrm{Ad}\,}}

\renewcommand{\ad}{\ensuremath{\mathrm{ad}\,}}

\newcommand{\diag}{\ensuremath{\mathrm{diag}}}

\newcommand{\Cl}{\ensuremath{\mathcal{C}}}

\newcommand{\SL}{\ensuremath{\mathrm{SL}}}

\newcommand{\un}{\ensuremath{\mathg{u}}}
\newcommand{\su}{\ensuremath{\mathg{su}}}
\newcommand{\SU}{\ensuremath{\mathrm{SU}}}
\newcommand{\U}{\ensuremath{\mathrm{U}}}

\newcommand{\so}{\ensuremath{\mathg{so}}}
\renewcommand{\sl}{\ensuremath{\mathg{sl}}}
\newcommand{\SO}{\ensuremath{\mathrm{SO}}}

\newcommand{\Spin}{\ensuremath{\mathrm{Spin}}}

\newcommand{\Orth}{\ensuremath{\mathrm{O}}}

\renewcommand{\k}{\ensuremath{\mathfrak{k}}}
\newcommand{\g}{\ensuremath{\mathfrak{g}}}
\newcommand{\h}{\ensuremath{\mathfrak{h}}}
\newcommand{\hol}{\ensuremath{\mathfrak{hol}}}
\newcommand{\m}{\ensuremath{\mathfrak{m}}}

\newcommand{\Iso}{\ensuremath{\mathrm{Iso}}}
\newcommand{\iso}{\ensuremath{\mathfrak{iso}}}
\renewcommand{\Re}{\ensuremath{\mathrm{Re}\,}}
\begin{document}
\def\haken{\mathbin{\hbox to 6pt{%
                 \vrule height0.4pt width5pt depth0pt
                 \kern-.4pt
                 \vrule height6pt width0.4pt depth0pt\hss}}}
    \let \hook\intprod
\setcounter{equation}{0}
%
%
\thispagestyle{empty}
%
\date{\today}
\title[The classification of naturally reductive 
spaces in dimensions $n \leq 6$]{The classification of 
naturally reductive homogeneous spaces in dimensions $n \leq 6$}
%
%
%
\author{Ilka Agricola}
\author{Ana Cristina Ferreira}
\author{Thomas Friedrich}
\address{\hspace{-5mm} 
Ilka Agricola,  
Fachbereich Mathematik und Informatik, 
Philipps-Universit\"at Marburg,
Hans-Meerwein-Stra\ss{}e,
D-35032 Marburg, Germany,
{\normalfont\ttfamily agricola@mathematik.uni-marburg.de}}
\address{\hspace{-5mm} 
Ana Cristina Ferreira,
Centro de Matem\'atica,
Universidade de Minho,
Campus de Gualtar,
4710-057 Braga, Portugal, 
{\normalfont\ttfamily anaferreira@math.uminho.pt}}
\address{\hspace{-5mm} 
Thomas Friedrich,
Institut f\"ur Mathematik,
Humboldt-Universit\"at zu Berlin,
Sitz: WBC Adlershof,
D-10099 Berlin, Germany,
{\normalfont\ttfamily friedric@mathematik.hu-berlin.de}}
%
\begin{abstract}
We present a new method for classifying naturally reductive homogeneous
spaces -- i.\,e.~homogeneous Riemannian manifolds admitting a metric connection
with skew torsion that has parallel torsion \emph{and} curvature.
This method is based on a deeper understanding of the holonomy algebra of
connections with parallel skew torsion on Riemannian manifolds  and the 
interplay of such a connection with the geometric structure on the given 
Riemannian manifold. It allows to reproduce by easier arguments 
the known classifications in dimensions $3,4$, and $5$, and yields
as a new result
the classification in dimension $6$. In each dimension, one obtains a 
`hierarchy' of degeneracy for the torsion form, which we then treat case 
by case. For the completely degenerate cases, we obtain results that
are independent of the dimension. In some situations, we are able to prove 
that any Riemannian manifold with parallel skew torsion has to be 
naturally reductive.  We show that a `generic'
parallel torsion form defines a quasi-Sasakian structure in dimension $5$
and an almost complex structure in dimension $6$. 
\end{abstract}
\maketitle
\pagestyle{headings}
%
%
%
\section{Introduction and summary}
%
The classification of Riemannian symmetric spaces by \'Elie Cartan in 1926
was one of his major contributions to mathematics, as it linked
in a unique way the algebraic theory of  Lie groups and the geometric
notions of  curvature, isometry, and holonomy. In contrast,
a classification of all Riemannian manifolds that are only homogeneous is,
without further assumptions, genuinely impossible. Substantial work
has been devoted to certain situations that are of particular interest;
for example, isotropy-irreducible homogeneous Riemannian spaces, 
homogeneous Riemannian spaces with positive curvature, and compact homogeneous
spaces of (very) small dimension have been investigated and classified.
The present paper is devoted to one such class of homogeneous
Riemannian manifolds---\emph{naturally reductive (homogeneous) spaces}. 
Traditionally, they are defined as  Riemannian manifolds $M=G/K$ with 
a transitive action of a Lie subgroup $G$ of the isometry group and
with a reductive complement $\m$ of the Lie subalgebra $\k$ inside
the Lie algebra $\g$ such that
\be\label{eq-cond-nat-red}
\langle [X,Y]_\m, Z\rangle + \langle Y, [X,Z]_\m\rangle\ =\ 0 \ 
\text{ for all } X,Y,Z\in\m.
\ee
Here,
 $\langle-,-\rangle$ denotes the inner product on $\m$ induced from 
the Riemannian metric $g$. 
Classical examples of naturally reductive homogeneous spaces
include: irreducible symmetric spaces, isotropy-irreducible 
homogeneous manifolds, Lie groups with a bi-invariant metric, and
Riemannian 3-symmetric spaces. The memoir 
\cite{DAtri&Z79} is devoted to the construction and, under some assumptions,
the classification of left-invariant naturally reductive
metrics on compact Lie groups.
Unfortunately, it is not always easy to decide whether a given homogeneous
Riemannian space is naturally reductive, since one has to consider
\emph{all} possible transitive groups of isometries.

Our study of naturally reductive spaces uses an alternative
description, and while doing so, enables us to obtain some general
structure results for them and another interesting larger class of manifolds,
namely \emph{Riemannian manifolds with parallel skew torsion}. 
As is well known, condition (\ref{eq-cond-nat-red}) states that the canonical connection
$\nabla^c$ of $M=G/K$  (see \cite{Kobayashi&N2}) has skew torsion 
$T^c\in \Lambda^3(M)$ (see Section
\ref{conn-with-skew-torsion} for precise definitions), and  a classical
result of Ambrose and Singer allows to conclude that the torsion $T^c$
and the curvature $\kr^c$ of $\nabla^c$ are then $\nabla^c$-parallel
\cite{Ambrose&S58}. For this reason, we agree to call a Riemannian manifold
$(M,g)$ \emph{naturally reductive} if it is a homogeneous space $M=G/K$
endowed with a metric connection $\nabla$ with skew torsion $T$ such that its
torsion and curvature $\kr$ are $\nabla$-parallel, 
i.\,e.~$\nabla T=\nabla \kr=0$.
If $M$ is connected, complete, and simply connected, a result
of Tricerri states that the space is indeed naturally reductive in
the previous sense {\cite[Thm 6.1]{Tricerri&V1}}. As a general reference 
for the relation between these notions, we recommend \cite{Tricerri93} 
and \cite{Cleyton&S04}. Appendix \ref{app.SHS} explains that 
any manifold  $M$ that is neither covered by a sphere 
nor a Lie group carries at most one connection that makes it naturally
reductive \cite{OR12b}. Where possible, we will first investigate
Riemannian manifolds with \emph{parallel skew torsion}, i.\,e.~we do not
assume, a priori, homogeneity or parallel curvature. 
Recall that for a Riemannian manifold with $K$-structure, a characteristic
connection is a $K$-connection with skew torsion (see 
\cite{FI02}, \cite{Agricola06}, and \cite{Agricola&F&H13} for uniqueness).
There are several classes of manifolds with a $K$-structure
that are  known to 
admit \emph{parallel} characteristic torsion, for example nearly K\"ahler 
manifolds, Sasakian manifolds, or nearly parallel $G_2$-manifolds; these 
classes have been considerably enlarged in
more recent  work (see \cite{Vaisman79},  \cite{FI02}, \cite{Alexandrov06}, 
\cite{Alexandrov&F&S05}, \cite{Friedrich07}, \cite{Sch07}), eventually 
leading to a host of such manifolds that are not homogeneous.
Nevertheless,  we will be able to prove that for some instances of 
manifolds with parallel skew  torsion, natural reductivity 
(and in particular, homogeneity) follows automatically.

The outline of our work is as follows. We recall some facts about
metric connections with torsion, their curvature, and
their holonomy. We introduce the $4$-form $\sigma_T$ that captures many
geometric properties of the original torsion form $T$; among others, it
is a suitable measure for the `degeneracy' of $T$. We prove 
splitting theorems for Riemannian manifolds with parallel skew torsion
if $\sigma_T=0$ (the `completely degenerate' case) and if $\ker T\neq0$; 
these will be used many times in the sequel. In fact, we are able to
show that an \emph{irreducible} manifold with parallel skew torsion $T\neq 0$
with $\sigma_T=0$ and dimension $\geq 5$ is a Lie group.
Then we turn to our main goal, namely, the classification of
naturally reductive spaces, and, sometimes, even
of spaces with parallel skew torsion, in dimension $\leq 6$.
The classification of naturally reductive  spaces was
obtained previously in dimension $3$ by Tricerri and Vanhecke
\cite{Tricerri&V1}, and in dimensions $4$, $5$ by Kowalski and Vanhecke
\cite{KV83}, \cite{KV85}. Their approach relied on  
deriving a normal form for the curvature tensor, and the torsion then
followed. With that approach, a classification in higher dimensions
was impossible, and the geometric structure of the manifolds was
not very transparent. In contrast, our approach is by looking at the
parallel torsion as the fundamental entity, and the general philosophy is
that, for sufficiently `non-degenerate' torsion, the connection
defined by the torsion is in fact the characteristic connection
of a suitable  $K$-structure on $M$. We make intensive use of the
fact that a $3$-form that is parallel for the connection it defines
is far from being generic, and that its holonomy has to be quite
special too: this allows a serious simplification of the situation.
Furthermore, we reformulate the first Bianchi identity as a single
identity in the Clifford algebra, a formulation that is computationally 
more tractable than its classical version. This is explained, together 
with  the closely related Nomizu construction, in Appendix  \ref{app.nomizu}.
We begin by rederiving the classifications by our alternative method
in dimensions $3$ and $4$. This turns to out to be very efficient, in
particular in dimension $4$---there, any parallel torsion $T\neq 0$ defines 
a parallel vector field $*T$ that induces a local splitting of 
the $4$-dimensional space. This generalizes and explains the classical result
in a few lines. 

In dimension $5$, the normal form of $T$ depends on two
parameters, which induce a case distinction (either $\sigma_T=0$ or
$\sigma_T\neq 0$, with subcases 
 $\Iso(T)=\SO(2)\x\SO(2)$ and  $\Iso(T)=\U(2)$). While the 
first case is immediately dealt with by our splitting theorems, the other
two are quite distinct geometrically. We prove that a  $5$-manifold with
parallel skew torsion is quasi-Sasakian in the second case and 
$\alpha$-Sasakian in the third case, with Reeb vector field $*\sigma_T$ (up 
to a constant).
In the second case, any such manifold is automatically naturally 
reductive, whereas non-homogeneous Sasaki manifolds are counter-examples for
the analogous statement in the last case. We finish the 
discussion with the full description of all naturally
reductive  $5$-manifolds; as a special case, we recover the
classification from \cite{KV85}. 

Our approach yields a complete classification of naturally reductive spaces 
in dimension $6$. The crucial observation is
that $*\sigma_T$ is a $2$-form, i.\,e.~an antisymmetric endomorphism
that can have rank 0, 2, 4, or 6. 
The ideal case  $\rk(*\sigma_T)=6$ and with all eigenvalues equal 
would define a K\"ahler form and thus the connection would be
the characteristic connection of an almost Hermitian structure.
We prove that this is basically what happens: the degenerate
case $\rk(*\sigma_T)=0$    is again easy. A $6$-manifold with 
parallel skew torsion and
 $\rk(*\sigma_T)=2$ is automatically naturally reductive and can be
parametrized explicitly. In particular, it is a product of two $3$-dimensional
non-commutative Lie groups equipped with a family of left-invariant
metrics. The case  $\rk(*\sigma_T)=4$ cannot occur, and
in case $\rk(*\sigma_T)=6$, the eigenvalues have to be indeed equal.
The space is either a nearly K\"ahler manifold or a family of 
naturally reductive 
spaces that can be described explicitly. The generic examples are 
left-invariant metrics 
on the groups $\SL(2,\C)$ or $\Spin(4) = S^3 \times S^3$.

For reference our paper includes a detailed section that describes all
homogeneous spaces appearing in the classification that are neither
products nor degenerate; we expect this part to be as useful as the
classification theorems themselves.

\textbf{Acknowledgments.} Ilka Agricola and Ana Ferreira 
acknowledge financial support by the
DFG within the priority programme 1388 "Representation theory".
Ana Ferreira thanks Philipps-Universit\"at Marburg for its
hospitality during a research stay in May-July 2013, and she also 
acknowledges partial financial support by the FCT through the 
project PTDC/MAT/118682/2010 and the University of 
Minho through the FCT project PEst-C/MAT/UI0013/2011. We also thank
Andrew Swann (Aarhus) for discussions on Section $4$ during a research visit
to Marburg in June 2013 and Simon G.~Chiossi 
(Marburg) for many valuable comments on a preliminary version of this work.
%
\section{Metric connections with skew torsion}\label{conn-with-skew-torsion}
%
Consider a Riemannian manifold $(M^n, g)$. The difference between its 
Levi-Civita connection $\nabla^g$ and any linear 
connection $\nabla$ is  a $(2,1)$-tensor field $A$,
\bdm 
\nabla_X Y\ =\ \nabla^g_X Y + A(X,Y),\quad X,Y \in TM^n.
\edm
The curvature of $\nabla$ resp.~$\nabla^g$ will 
always be denoted by $\kr$ resp.~$\kr^g$.
Following Cartan, we study the algebraic 
types of the torsion tensor for a metric connection.
Denote by the same symbol the $(3,0)$-tensor
derived from a $(2,1)$-tensor  by contraction with the metric.
We identify $TM^n$ with $(TM^n)^*$ using $g$ from now on. 
Let $\mathcal{T}$ be the $n^2(n-1)/2$-dimensional space of all 
possible torsion tensors,
\bdm
\mathcal{T}\ =\ \{T\in\ox^3 TM^n \ | \ T(X,Y,Z)= - \, T(Y,X,Z) \}
\ \cong \ \Lambda^2(M^n)\ox TM^n \, .
\edm
A connection $\nabla$ is metric if and only if $A$ belongs to the space
\bdm
\mathcal{A}\ :=\ TM^n\ox\Lambda^2(M^n) \ = \ \{A \in\ox^3 TM^n \ | 
\ A(X,V,W) +  A(X,W,V) \ =\ 0\} \, .
\edm 
The spaces $\mathcal{T}$ and $\mathcal{A}$ are isomorphic
as $\Orth(n)$ representations,  reflecting
the fact that metric connections can be uniquely characterized by their
torsion.
For $n\geq 3$,  they split under the action of $\Orth(n)$
into the sum of three irreducible representations,
\bdm
\mathcal{T}\cong TM^n \op \Lambda^3(M^n) \op \mathcal{T}'.
\edm
The last module  is equivalent to the
Cartan product  of the representations  $TM^n$ and $\Lambda^2(M^n)$ 
(see \cite{Cartan25a}).
The eight classes of linear connections are now defined by the possible
components of their torsions $T$ in these spaces.
The main case of interest is the following:
\begin{dfn}
The connection $\nabla$ is said to have  \emph{skew-symmetric 
torsion} or just \emph{skew torsion} if its torsion tensor lies in  
the second component of the decomposition, i.\,e.~it is
given by a $3$-form,
\be\label{eq-norm-torsion}
\nabla_X Y \ = \ \nabla^g_X Y \, + \, \frac{1}{2} \, T(X,Y,-).
\ee 
Whenever we shall speak of a metric connection with skew torsion
$T\in\Lambda^3(M^n)$, the connection is defined with the normalization as in 
equation $(\ref{eq-norm-torsion})$.
By a \emph{Riemannian  manifold with  parallel skew torsion}, we mean
a Riemannian manifold $(M^n,g)$ equipped with a metric connection $\nabla$ 
with skew torsion $T$ such that $\nabla T=0$. 
For any $3$-form $T\in\Lambda^3(M^n)$, one can define a $4$-form through
\bdm
\sigma_T \ :=\ \frac{1}{2} \sum_{i=1}^n (e_i\haken T)\wedge (e_i\haken T),  \
\mbox{or equivalently} \ \sigma_T(X,Y,Z,V) \ =\ \cyclic{X,Y,Z}
g(T(X,Y), T(Z,V)),
\edm
where $e_1,\ldots,e_n$ is a local orthonormal frame.  
It clearly defines a covariant from $\Lambda^3(\R^n)$ to 
$\Lambda^4(\R^n)$ for the
standard $\SO(n)$-action. Furthermore, this $4$-form is      an 
important measure for the
`degeneracy' of $T$ and appears in many formulas, like the identities on 
curvature and derivatives below,  
or the Nomizu construction (Appendix \ref{app.nomizu}). It also
describes the algebraic action of any $2$-form   $X\haken T$, 
identified with an element of $\so(n)$,  on $T$
(see also \cite[proof of Prop. 2.1]{Agricola&F10}),
\be\label{eq.XT-auf-T}
(X\haken T)(T)(Y_1,Y_2,Y_3)\ =\ \sigma_T(Y_1,Y_2,Y_3,X).
\ee
\end{dfn}
The condition that $\nabla$ is metric implies
that each curvature transformation $\kr(X,Y)$ is skew-adjoint,
i.\,e.~$\kr$ can be interpreted  as an endomorphism
$\kr:\ \Lambda^2(M^n)\ra \hol^{\nabla} \subset \Lambda^2(M^n)$ where 
$\hol^{\nabla}$ is the holonomy algebra of $\nabla$. If, in addition, 
$\nabla T=0$,
then $\kr$ is a symmetric endomorphism,
\bdm
g(\kr(X,Y)V,W)\ =\ g(\kr(V,W)X,Y),
\edm
and the Ricci tensor $\Ric$ is symmetric too.
Consequently, the curvature operator $\kr = \ \psi\circ \pi$ is 
given by the projection
$\pi$ onto the holonomy algebra and a symmetric endomorphism 
$\psi : \hol^{\nabla} \ra \hol^{\nabla}$. If moreover $\psi$ is 
$\hol^{\nabla}$-equivariant, then
$\kr$ is $\nabla$-parallel, i.e.\,$T$ and $\kr$ define a naturally 
reductive structure.
For parallel skew torsion, the first
Bianchi identity may be written as \cite{Kobayashi&N1}
\be\label{Bianchi-I} 
\cyclic{X,Y,Z}\kr(X,Y,Z,V)\ =\  \sigma_T(X,Y,Z,V),
\ee
and the following remarkable relations hold \cite{FI02},
\be\label{thm.first-cons}
\nabla^g T \ = \ \frac{1}{2}\sigma_T  , \quad  
dT \ = \ 2\sigma_T .
\ee 
If  the torsion \emph{and} the curvature $\kr$ of $\nabla$ are
$\nabla$-parallel, the second Bianchi identity is reduced to the algebraic
relation (\cite{Kobayashi&N1})
\be\label{Bianchi-II}
\cyclic{X,Y,Z} \kr (T(X,Y),Z)\  =\ 0. 
\ee
These identities can be nicely formulated in the Clifford algebra (see
Theorem \ref{Bianchi-Clifford}) and we will use this approach several times in our dicussion.
Last but not least, let us recall that
any $\nabla$-parallel vector field $V$ or $2$-form $\Omega$ 
satisfies the differential equations (see \cite{Agricola&F04a})
\be\label{par-2-form}
\langle \nabla^g_X V,Y\rangle = - T(X,V,Y)/2 \ , \quad \mbox{or} \quad 
\delta^g\Omega\ =\ \frac{1}{2} \Omega\haken T,\quad
d\Omega\ =\ \sum_{i=1}^n (e_i\haken \Omega)\wedge (e_i\haken T).
\ee
In particular, any $\nabla$-parallel vector field is a Killing 
vector field of constant length.
A parallel torsion form is of special algebraic type. 
Indeed, in this case the holonomy algebra $\hol^\nabla$ is contained in the 
isotropy algebra $\iso(T)$ of $T$. Here are some useful observations on $\hol^\nabla$.
\begin{prop}\label{prop.sigma1}
Let $(M^n,g, T)$ be a Riemannian manifold with parallel skew torsion. Then 
the following properties hold:
\begin{enumerate}
\item Either $\sigma_T = 0$ or the algebra $ \hol^{\nabla} \subset \iso(T)$ 
is non-trivial.
\item If $\dim \hol^\nabla \leq 1$, then there exists a $2$-form $\omega$ such that $\sigma_T  =  \pm \, \omega \wedge \omega$.
\item If $\hol^\nabla$ is abelian, then $\nabla\kr =0$ and therefore 
$M^n$ is naturally reductive.
\end{enumerate}
\end{prop}
\begin{proof}
The first claim follows  from the first Bianchi identity (\ref{Bianchi-I}). 
For the second, let $\omega \in  \hol^\nabla \subset \Lambda^2$ be a generator. 
Since $\nabla T = 0$, 
the symmetric curvature operator $\kr : \Lambda^2 \rightarrow \Lambda^2$ is 
the projection onto $\hol^\nabla$, $\kr  =  a \, \omega \odot \omega$.
Inside the Clifford algebra $\mathcal{C}$ the Bianchi identity 
reads as (see Appendix A)
\bdm
- \, 2 \, \sigma_T  +  \|T\|^2  + \kr \ = \ T^2 + \kr  \in  \R \ \subset \
\mathcal{C}. 
\edm
We compare the parts of order four in $\mathcal{C}$ and obtain the result. 
For the last assertion, we write again the  curvature operator as
$\kr= \psi\circ \pi$, where $\pi$ denotes the projection
onto the holonomy algebra $\hol^\nabla$
and $\psi : \hol^{\nabla} \ra \hol^{\nabla}$ is
a symmetric endomorphism.
As a projection, $\pi$ is $\hol^\nabla$-equivariant,
and $\hol^\nabla$ acts by the adjoint action on itself -- and this action is
trivial, since $\hol^\nabla$ is abelian. Therefore, $\psi$ is trivially
$\hol^\nabla$-equivariant, and so this holds for the composition $\kr$. 
This shows $\nabla\kr=0$ by the general holonomy principle.
\end{proof}
%
%
%
\section{Splitting properties}
%
%
Splitting theorems allow us to identify situations
in which a naturally reductive space has to be a product. 
We will use the following well-known statement.
\begin{thm}[{\cite{KV83}}]\label{thm.prod-nat-red}
The Riemannian product $(M_1,g_1)\x (M_2,g_2)$ is naturally reductive
if and only if both factors $(M_i,g_i), \,  i=1,2$ are naturally reductive.
\end{thm}
The following two notions turn out to be appropriate for analyzing $3$-forms.
\begin{dfn}\label{kerT-algebra-gT}
For any $3$-form $T\in\Lambda^3(M^n)$, we define the kernel 
\bdm
\ker T \ :=\ \{ X\in TM^n\, |\, X\haken T = 0\},
\edm
and the Lie algebra generated by the image, 
\bdm
\g_T\ :=\ \mathrm{Lie}\langle X\haken T \, |\, X\in TM^n \rangle.
\edm
The Lie algebra $\g_T$  was first considered in  \cite{Agricola&F04a} and 
is \emph{not}
related in any obvious way to the isotropy algebra of $T$.
If $\g_T$ does not act irreducibly on $TM^n$, 
it is known that it splits into an orthogonal sum $V_1\oplus V_2$ of
$\g_T$-modules
and $T=T_1+ T_2$ with $T_i\in\Lambda^3 (V_i)$ \cite[Prop.3.2]{Agricola&F04a};
thus $V_1$ and $V_2$  are  $\nabla$-parallel
subbundles of $TM^n$. If $\sigma_T=0$, equation (\ref{thm.first-cons}) 
and de Rham's Theorem imply:
\end{dfn}
\begin{thm}\label{cor.splitting-g-T}
Let $(M^n,g,T)$ be a complete, simply connected Riemannian manifold
with parallel skew torsion $T$ such that $\sigma_T=0$, and let 
$TM^n= \mathcal{T}_1\oplus \ldots \oplus \mathcal{T}_q$ be the decomposition of $TM^n$ into 
$\g_T$-irreducible, $\nabla$-parallel distributions. Then all 
distributions $\mathcal{T}_i$
are $\nabla^g$-parallel integrable distributions,  
$M^n$ is a Riemannian product, and the torsion
$T$ splits accordingly,
\bdm
(M^n,g,T)\ =\ (M_1,g_1,T_1)\x \ldots \x(M_q, g_q, T_q), \quad T \ = \ \sum_{i=1}^q T_i \, .
\edm
\end{thm}
The following new splitting theorem for manifolds with
parallel skew torsion justifies why we will consider in the following
only connections with parallel skew torsion whose kernel is trivial.
Let us emphasize that the vanishing of $\sigma_T$ is not needed for it.
\begin{thm}\label{thm.splitting-with-kernel}
Let $(M^n,g,T)$ be a complete, simply connected Riemannian manifold
with parallel skew torsion $T$. Then $\ker T$ and $(\ker T)^\perp$
are  $\nabla$-parallel and $\nabla^g$-parallel integrable distributions.
Furthermore, $M^n$ is a Riemannian product such that $T$ vanishes
on one factor and has trivial kernel on the other,
\bdm
(M^n,g,T)\ =\ (M_1,g_1,T_1=0)\x (M_2, g_2, T_2),\quad \ker T_2\, =\, \{ 0\}.
\edm
\end{thm}
\begin{proof}
%
If $T$ is parallel for the connection $\nabla$ it defines, 
$\ker T$ and $(\ker T)^\perp$ are clearly $\nabla$-parallel distributions.
Suppose $X$ is a vector field in
$\ker T$.  Then, for any vector field $Y$, 
\bdm
\nabla^g_Y X\ =\ \nabla_Y X -\frac{1}{2} T(Y,X) \ =\ 
\nabla_Y X 
\edm
belongs again to $\ker T$. Consequently, $\ker T$ is  $\nabla^g$-parallel
as well. Since $\nabla$ and $\nabla^g$ are metric, the same holds
for $(\ker T)^\perp$, and both distributions are integrable.
The last assertion follows then from de Rham's Theorem.
\end{proof}
In particular, this applies to the following situation:
given a simply connected naturally reductive space with  
torsion $T$ that does not depend on all directions, the space is a
product of lower dimensional naturally reductive homogeneous spaces.
%
%
\section{Vanishing of $\sigma_T$}
%
Let $(M^n,g)$ be a Riemannian manifold with parallel skew torsion $T$:
we are interested in describing those manifolds for which $\sigma_T=0$.
Proposition \ref{prop.sigma1} and Theorem 
\ref{cor.splitting-g-T} suggest that the vanishing of $\sigma_T$ is a strong
indication that the geometry is degenerate. Let us give some examples
where this happens. In dimension $3$, the volume form
of any metric connection $\nabla$ is $\nabla$-parallel, and $\sigma_T=0$
for trivial reasons: hence, any Riemannian $3$-manifold is an
example of such a geometry. In dimension $4$, again $\sigma_T=0$
for algebraic reasons, so the condition is vacuous. 
Finally, compact
Lie groups with a bi-invariant metric and any connection from the
Cartan-Schouten family of connections are well-known examples
\cite[Proposition 2.12]{Kobayashi&N2}. We shall now prove:
\begin{thm}\label{thm.sigmaT-vanishes}%
Let $(M^n,g)$ be an irreducible, complete, and simply connected
 Riemannian  manifold with
parallel skew torsion $T\neq 0$ such that $\sigma_T=0$, $n\geq 5$. Then 
$M^n$ is a simple compact Lie group with bi-invariant metric or its dual
noncompact symmetric space.
\end{thm}
\begin{proof}
Let $\hol^g, \, \hol^\nabla$, and 
$\iso(T)$ be the Lie algebras of the holonomy groups of $\nabla^g, \, \nabla$
and the Lie algebra of the isotropy group of $T$, respectively.
By equation (\ref{thm.first-cons}),  $\nabla^g T=\sigma_T / 2 =0$, hence
both $\hol^\nabla$ and $\hol^g$ are subalgebras of $\iso(T)$.
This implies also that any multiple of $T$ defines again
a connection with parallel skew torsion and vanishing $4$-form.
Fix some point $p\in M^n$, $V:=T_p M$. The Nomizu construction
(Lemma \ref{lem.nomizu}) yields for $\sigma_T=0$ that $T$ defines a 
Lie algebra structure on
$V$ by $g_p([X,Y],Z)= \alpha \, T_p(X,Y,Z)$ for any $ \alpha\in\R$.
By assumption, $\hol^g$ acts irreducibly on $V$.
Consider the Lie algebra $\g_T$ generated by all 
elements of the form $X\haken T$, see Definition \ref{kerT-algebra-gT}.
Since $\so(n)\cong \Lambda^2 (V)$ acts on $V$, it also acts on $\Lambda^3 (V)$, 
and for an element $X\haken T \in \Lambda^2 (V)$ this action on $T$ is
 given by 
(see equation (\ref{eq.XT-auf-T}))
\bdm
(X\haken T)(T)(Y_1,Y_2,Y_3)\ =\ \sigma_T(Y_1,Y_2,Y_3,X)\ = \ 0.
\edm
Hence, $\g_T\subset \iso(T)$ as well. We can
assume that $\g_T$ acts irreducibly on $V$ by Theorem 
\ref{cor.splitting-g-T}. We conclude that $(G_T, V,T)$ defines an irreducible
skew torsion holonomy system in the sense of Definition \ref{dfn.SHS},
where $G_T$ denotes the Lie subgroup of $\SO(n)$ with Lie algebra $\g_T$.
By the Skew Holonomy Theorem of Olmos and Reggiani (Theorem \ref{thm.SHS}),
 $(G_T, V,T)$ is either transitive or symmetric. 
If it is transitive,
$G_T=\SO(n)$, hence $\g_T=\so(n)=\iso(T)$. But for $n\geq 5$, this implies
$T=0$, a contradiction.  Thus,
the skew holonomy system is symmetric, i.\,e.~$G_T$ is simple, acts on $V$
by its adjoint action, and $T$ is, up to multiples, unique. 

Next, we observe that the algebras $\g_T$ and $\iso(T)$ coincide. Indeed,
any element of $\iso(T)$ is a derivation of the simple Lie algebra $\g_T$.
But all derivations of $\g_T$ are inner, i.\,e.~defined by elements
of $\g_T$. As a consequence, the holonomy algebra $\hol^g$ is also contained
in $\g_T$. If they were not equal, then we could decompose
the tangent space into $V=\g_T= \hol^g\oplus (\hol^g)^\perp$, which
contradicts the assumption that $M$ is an irreducible Riemannian manifold.
To summarize, $\iso(T)= \g_T = \hol^g$.
By Berger's
holonomy theorem, $M^n$ is an irreducible Riemannian symmetric space, and
its isotropy representation of $\hol^g$ is its adjoint representation.
Certainly, the compact simple Lie group $G_T$ itself, viewed as a symmetric
space (or its dual noncompact space), is such a manifold. Since a
symmetric space is uniquely characterised by its isotropy representation, \cite[p.427, Theorem 5.11]{HelgasonBook},
other symmetric spaces cannot occur.
\end{proof}
\begin{exa}
Let $T$ be a $3$-form with constant coefficients on  $\R^n$
satisfying $\sigma_T=0$. Then the flat space $(\R^n,g, T)$ is a 
reducible Riemannian
manifold with parallel skew torsion and $\sigma_T=0$. This shows that the 
assumption that $M$ be irreducible is crucial in this result.
Observe that in this example, the Riemannian manifold is decomposable,
but the torsion is not.
\end{exa}
%
%
\section{The classification in dimension $3$}\label{subsec:dim-3}
%
Let $(M^3,g)$ be a complete Riemannian $3$-manifold admitting a metric 
connection $\nabla$ with parallel torsion $T\neq 0$.
Since $\|T\|=\mathrm{const}$, $T$ defines a nowhere vanishing $3$-form and
hence $M^3$ is orientable. Denoting the volume form by $dM$, we conclude
$T=\lambda dM$ for some constant $\lambda\neq 0$.

We begin by giving a proof of the classification of $3$-dimensional
naturally reductive spaces that illustrates the methods that we will use
later. This result is due to Tricerri
and Vanhecke \cite{Tricerri&V1}, but with another approach.
A detailed description of naturally reductive metrics on
$\SL(2,\R)$ (case 2b in the theorem below) may be found 
in \cite{Halverscheid&I06}.
\begin{NB}\label{notations}
For ease of notation, we will often omit the wedge product between
$1$-forms $e_i, e_j\ldots$, i.\,e.~$e_{ijk}:=e_i\wedge e_j\wedge e_k$,
like in the following proof. Furthermore, we will freely identify
$\so(n)$ and $\Lambda^2(\R^n)$ where necessary, i.\,e.~there is no
difference between the $2$-form $e_{ij}$ and the endomorphism $E_{ij}$
mapping $e_i$ to $e_j$, $e_j$ to $-e_i$ and eveything else to zero.   
\end{NB}
%
\begin{thm}[{\cite[Thm 6.5, p.\,63]{Tricerri&V1}}]\label{thm.class-dim3}
A  three-dimensional complete, simply connected  naturally 
reductive  space $(M^3,g)$ is either
\begin{enumerate} 
\item a space form: $\R^3, S^3$ or $\mathbb{H}^3$, or
\item isometric to one of the following Lie groups with a 
suitable left-invariant metric:
\begin{enumerate}
\item the special unitary group $\SU(2)$,
\item $\widetilde{\SL}(2,\R)$, the universal cover of $\SL(2,\R)$, 
\item the $3$-dimensional Heisenberg group $H^3$ (see section 9.3).
\end{enumerate}
\end{enumerate}
\end{thm}
\begin{proof}
%
Indeed, the Ambrose-Singer Theorem  tells us that $(M^3,g)$ is homogeneous 
and the Lie algebra of the group $G$ which acts transitively and effectively 
on $M^3$ is isomorphic to the direct sum $\g = \k \oplus \m$, where $\k$ is 
the holonomy algebra of $\nabla$ and $\m$ is the tangent space at some 
point. If $M^3$ is Einstein, it is isometric to $\R^3, S^3$, or $H^3$. If the
Riemannian 
Ricci tensor $\Ric^g$ has distinct eigenvalues, $\k$ is at most one-dimensional.
Indeed, $\Ric $ is $\nabla$-parallel, and the difference
$\Ric^g-\Ric$ is given by $T$, which is also $\nabla$-parallel. Hence $\nabla\Ric^g = 0$
and $\Ric^g$ is, by assumption, not a multiple of the idendity. Therefore  $\k$  cannot
coincide with $\so(3)$.
Let us apply the Nomizu construction in this situation 
(Appendix \ref{app.nomizu}). The one-dimensional Lie algebra $\k$ is generated
by the $2$-form $\Omega=:e_{12}$. The curvature operator is symmetric
(see Section \ref{conn-with-skew-torsion}) 
and consequently a multiple of the projection onto the
subalgebra $\k$, i.\,e.
\bdm
T\ =\ \lambda\ e_{123} \ \text{ and }\ \kr\ =\ \alpha \, e_{12}\odot e_{12}.
\edm
By the Nomizu construction, $e_1, e_2, e_3$, and $\Omega$ are a basis
of $\g$ with non-trivial Lie brackets
%
%
\bdm
[e_1,e_2]=-\alpha\Omega-\lambda e_3=:\tilde\Omega,\quad 
[e_1,e_3] = \lambda e_2,\quad [e_2,e_3]= -\lambda e_1,\quad
[\Omega, e_1]=e_2,\quad [\Omega, e_2] = -e_1.
\edm
The $3$-dimensional subspace $\h$ spanned by $e_1, e_2$, and $\tilde\Omega$
is a Lie subalgebra of $\g$ that is transversal to the isotropy algebra $\k$
(since $\lambda\neq 0$). Consequently, the corresponding simply connected
Lie group $H$ acts transitively on $M^3$, i.\,e.~$M^3$ is a Lie group with a
left-invariant metric. One checks that $\h$ has the commutator relations
\bdm
[e_1,e_2]= \tilde\Omega,\quad
[\tilde\Omega,e_1]= (\lambda^2-\alpha)e_2,\quad
[e_2,\tilde\Omega]= (\lambda^2-\alpha)e_1.
\edm
For $\alpha=\lambda^2$, this is the $3$-dimensional Heisenberg Lie algebra,
otherwise it is $\su(2)$ or $\sl(2,\R)$ depending on the sign of 
$\lambda^2-\alpha$.
\end{proof}
\section{The classification in dimension $4$}\label{subsec:dim-4}
%
The classification of naturally reductive spaces 
below appeared first in \cite{KV83}. 
The original proof used a 
much more involved curvature calculation. Here we present a much simpler proof.
An easy algebraic argument shows that $\sigma_T=0$
for any $3$-form $T$ in dimension $4$. Therefore:
\begin{thm}
Let $(M^4,g, T)$ be a complete, simply connected Riemannian $4$-manifold 
 with parallel skew torsion $T\neq 0$.
Then 
\begin{enumerate}
\item 
$V:=*T$ is a $\nabla^g$-parallel vector field.
\item
The Riemannian holonomy algebra $\hol^g$ is contained in 
$\so(3)$, and hence  $M^4$ is 
isometric to a product $N^3\x \R$, where $(N^3,g)$ is a
$3$-manifold with a parallel $3$-form $T$.
\end{enumerate}
\end{thm}
Since any $3$-form in dimension $4$ has a $1$-dimensional kernel, the result follows directly from Theorem \ref{thm.splitting-with-kernel}. However, 
we can also give a direct geometric proof.
\begin{proof}
Any vector field $V$ satisfies $V\haken (* V)=0$. But  $*V=T$, so 
$V\haken T=0$. $\nabla V=0$ is equivalent to
$\nabla^g_X V = - T(X,V,-)/2$, and by the previous observation,
the right hand side vanishes, so $\nabla^g_X V=0$. The vector field $V$
is complete, because it is a Killing vector field on a complete Riemannian
manifold.
Since the stabilizer of a vector field is $\SO(3)\x\{1\}$ (acting on 
the orthogonal complement of $V$), the holonomy claim follows at once
from the general holonomy principle. The isometric splitting is then a
consequence of de Rham's Theorem.
\end{proof}
Thus, Theorem \ref{thm.prod-nat-red} implies at once:
\begin{cor}[\cite{KV83}]
A $4$-dimensional simply connected naturally reductive Riemannian 
manifold with $T\neq 0$ is
isometric to a Riemannian product  $N^3\x \R$, where $N^3$
is a $3$-dimensional  naturally reductive Riemannian manifold.
\end{cor}
%
%
\section{The classification in dimension $5$}
%
The following lemma follows immediately from the well-known normal
forms for $3$-forms in five dimensions.
\begin{lem}\label{lem.dim5-algebra}
Let $(M^5,g, T)$ be an orientable  
Riemannian $5$-manifold with parallel skew torsion $T\neq 0$. Then 
there exists a local orthonormal frame $e_1,\ldots,e_5$ such that
\be\label{dim5-torsion}
T\ =\ -(\vrho e_{125}+\lambda e_{345}), \quad
* T\ =\ -(\vrho e_{34}+\lambda e_{12}),\quad \sigma_T\ =\ \vrho\lambda e_{1234}
\ee
for two real constants $\vrho,\lambda$.
The isotropy group and its action on the tangent space at any $x\in M$
is:

\smallskip
\bdm
\ba{|l|l|l|l|} \hline
\text{Case A} &  \sigma_T= 0 & \vrho \lambda = 0&\Iso(T)= \SO(3)\x\SO(2) \\ \hline
\text{Case B.$1$} &  \sigma_T\neq 0 & \vrho \lambda \neq 0, \vrho\neq \lambda  &
\Iso(T)= \SO(2)\x\SO(2)\x\{1\} \\ \hline
\text{Case B.$2$} &  \sigma_T\neq 0 & \vrho \lambda \neq 0, \vrho=\lambda  &
\Iso(T)= \U(2)\x\{1\}  \\ \hline
\ea
\edm
\smallskip
\end{lem}
The degenerate case A ($\sigma_T=0$) follows 
from Theorem \ref{thm.splitting-with-kernel}:
\begin{prop}[case A -- with parallel skew torsion]
Let $(M^5,g, T)$ be a complete, simply connected
Riemannian $5$-manifold with parallel skew torsion $T\neq 0$ 
such that $\sigma_T=0$. Then $M^5$ is isometric to 
a product $N^3\x N^2$, where $N^3$ is a  Riemannian $3$-manifold
with torsion $T=c\, dN^3 $, a multiple of the volume form, and $N^2$ is any
Riemannian  $2$-manifold.
\end{prop}
Observing that the curvature of a surface is its Gaussian curvature,
Theorem \ref{thm.prod-nat-red} implies:
\begin{cor}[case A -- the naturally reductive case]
A simply connected naturally reductive Riemannian 
$5$-manifold with $T\neq 0,\ \sigma_T=0$ is
isometric to a Riemannian product  $N^3\x N^2$, where $N^3$
is a $3$-dimensional  naturally reductive Riemannian manifold and $N^2$
a surface of constant Gaussian curvature.
\end{cor}
Recall that a Sasakian manifold  always has parallel characteristic torsion,
and that many non-homogeneous Sasakian manifolds are known
(in all odd dimensions). Consequently, a classification of $5$-manifolds
with parallel skew torsion is not possible for  $\lambda=\vrho$ (case B.2).
In contrast to this, our next result shows that   $5$-manifolds with parallel 
skew torsion and $\lambda \neq\vrho$ (case B.1) are necessarily naturally
reductive and can be completely described:
\begin{thm}[case B.1 -- with parallel skew torsion]\label{b1-immer-natred}
Let $(M^5,g, T)$ be an orientable  
Riemannian $5$-manifold with parallel skew torsion and such that 
$T$ has the normal form
\bdm
T\ =\ -(\vrho e_{125}+\lambda e_{345}), \quad
\vrho\lambda \neq 0 \text{ and } \vrho\ \neq\ \lambda.
\edm
Then $\nabla \kr =0$.
The admissible torsion forms and curvature operators
depend on $4$ parameters. Moreover, if $M^5$ is complete, then it is 
locally naturally reductive.
\end{thm} 
\begin{proof}
The curvature transformation
is a symmetric endomorphism $\kr:\ \Lambda^2 (M^5)\ra \Lambda^2 (M^5) $
(see Section \ref{conn-with-skew-torsion}).
We consider the $2$-forms $\Omega_1:=e_{12}$ and  $\Omega_2 :=e_{34}$,
of which we know that they are closed and $\nabla$-parallel
(see the proof  of Theorem \ref{thm-B-contact}). We split 
\bdm
\so(5)\ \cong\ \Lambda^2(M^5)\ =\ \Lin(\Omega_1,\Omega_2)\oplus 
\Lin(\Omega_1,\Omega_2)^\perp\ =: \ \h\oplus \h^\perp, 
\edm
and observe that the first summand is not only a subspace, 
but a two-dimensional 
abelian subalgebra $\h\cong \so(2)\oplus \so(2)$ -- in fact, by 
Lemma \ref{lem.dim5-algebra},  $\hol^{\nabla}\subset \h$.
This last property  implies that the curvature $\kr$ has image inside
$\h$. In particular, it means that $\hol^{\nabla}$ is abelian too and
therefore $\nabla\kr=0$ by Proposition \ref{prop.sigma1} (3).
Hence, there exist constants $a,b,c$ such that
\bdm
\kr\ =\ a\, \Omega_1\odot \Omega_1 + b\, \Omega_1\odot \Omega_2+
c\, \Omega_2\odot\Omega_2.
\edm
One sees that inside the Clifford algebra,
$(\kr + T^2)_4 =(b -2\lambda\vrho)\, e_1 e_2 e_3 e_4$, hence
the  first Bianchi identity in the Clifford version 
(Theorem \ref{Bianchi-Clifford}) yields $b=2\lambda\vrho$.
\end{proof}
\begin{NB}
A routine computation shows that the Ricci tensor of $\nabla$ is then
given by
\bdm
\Ric\ =\ \diag(-a,-a,-c,-c,0).
\edm
In particular, $a=c=0$ yields a $\nabla$-Ricci-flat space.
\end{NB}
We shall now describe explicitly the naturally reductive spaces 
and the group $G$ acting transitively on them covered by the previous
theorem. As a vector space, its Lie algebra $\g$ is given
by the Nomizu construction as $\g=\h\oplus \R^5$ with basis
$\Omega_1, \Omega_2, e_1, \ldots, e_5$. Using the explicit formulas
for $T$ and $\kr$, one computes the commutator relations
\begin{gather*}
[e_1,e_2]\, =\, - a\Omega_1-\lambda\vrho\,\Omega_2+\vrho\, e_5\, =:\,
\tilde\Omega_1,\quad 
[e_3,e_4]\, =\, - \lambda\vrho\,\Omega_1-c\,\Omega_2+\lambda\, e_5\, =:\,
\tilde\Omega_2,\\
[e_1,e_5]\, =\, -\vrho\, e_2,\quad 
[e_2,e_5]\, =\, \vrho\, e_1,\quad
[e_3,e_5]\, =\, -\lambda\, e_4,\quad
[e_4,e_5]\, =\, \lambda \, e_3,\\
[\Omega_1, e_1]\, = \, +e_2,\quad
[\Omega_1, e_2]\, = \, -e_1,\quad
[\Omega_2, e_3]\, = \, +e_4,\quad
[\Omega_2, e_4]\, = \, -e_3.
\end{gather*}
The crucial observation is that the linear space spanned by $e_1, e_2, e_3,
e_4,\tilde\Omega_1, \tilde\Omega_2$ is a Lie subalgebra that we will denote
by $\g_1$. Its non-vanishing Lie brackets are
\begin{gather*}
[e_1,e_2]\, =\,\tilde\Omega_1,\quad
[\tilde\Omega_1,e_1]\, =\, (\vrho^2-a)\, e_2,\quad
[e_2, \tilde\Omega_1]\, =\, (\vrho^2-a)\, e_1,\\ 
[e_3,e_4]\, =\,\tilde\Omega_2,\quad
[\tilde\Omega_2,e_3]\, =\, (\lambda^2-c)\, e_4,\quad
[e_4, \tilde\Omega_2]\, =\, (\lambda^2-c)\, e_3.
\end{gather*}
Define
\bdm
h_1\ := \ \lambda\, \tilde\Omega_1-\vrho\,\tilde\Omega_2\ =\ 
\lambda(\vrho^2-a)\,\Omega_1+ \vrho (c-\lambda^2)\, \Omega_2.
\edm
There are two cases to consider.
If $a=\vrho^2$ and $c=\lambda^2$, $\tilde\Omega_1$ and $\tilde\Omega_2$
are proportional, therefore $\g_1$ is $5$-dimensional, isomorphic to the
$5$-dimensional Heisenberg algebra, and $\g_1\cap \h$ is
trivial (since $\lambda \neq 0$). We conclude that the corresponding
Heisenberg group $H^5$ acts transitively on the $5$-manifold $M$.
Consequently, $(M,g)$ is isometric to a left-invariant metric 
on $H^5$, and the metric depends on two parameters.  

In the second case, either $a\neq\vrho^2$ or $c\neq\lambda^2$.
The intersection $\g_1\cap \h =:\h_1$ is $1$-dimensional and generated by
the element $h_1\neq 0$.
By the commutator equations above,
the Lie algebra $\g_1$ is a direct sum of two $3$-dimensional ideals,
and the last formula describes how $\h_1$ is embedded in this direct
sum decomposition of $\g_1$. If we assume that the coefficients are chosen
in such a way that $h_1$ generates a \emph{closed} subgroup of $G_1$,
we can assert $M$ is isometric to $G/H\cong G_1/H_1$.
Furthermore, each of these ideals is isomorphic to $\so(3), \sl(2,\R)$, 
or to the
$3$-dimensional Heisenberg algebra, depending on the value
of the constants $\vrho^2-a$ and $c-\lambda^2$. 
We emphasize that the case of two
$3$-dimensional Heisenberg algebras is excluded: it would
correspond to the vanishing of both constants, which was
discussed earlier. 
We summarize this discussion in the following theorem:
\begin{thm}[case B.1 -- classification]\label{thm.B1-classification}
A complete, simply connected Riemannian $5$-manifold  $(M^5,g, T)$ 
with parallel 
skew torsion $T$ such that
$T = -(\vrho e_{125}+\lambda e_{345})$ 
with $\vrho\lambda \neq 0, \vrho\ \neq\ \lambda$ is one of the
following homogeneous spaces:
\begin{enumerate}
\item The $5$-dimensional Heisenberg group $H^5$  with a two-parameter 
family of left-invariant metrics (described in Section 
\ref{odd-Heisenberg}),
\item A manifold of  type $(G_1\x G_2)/\SO(2)$ where
$G_1$ and $G_2$ are either $\SU(2)$,  $\widetilde{\SL(2,\R)}$, or $H^3$,
but not both equal to $H^3$ 
with one parameter $r\in \Q$ classifying the embedding of $\SO(2)$
and  a three-parameter family of homogeneous metrics (described 
in Section \ref{Stiefel-type}).
\end{enumerate}
\end{thm}
\begin{NB}\label{NBIsotropy}
Theorem \ref{thm.B1-classification} gives an isometric description with a 
one-dimensional isotropy group of the homogeneous spaces under consideration.
The description as naturally reductive spaces uses a two-dimensional
isotropy group. 
\end{NB}
Let us now discuss the case B.2, i.\,e. $\lambda=\vrho$.
As observed before, there exist manifolds of type B.2
that have parallel skew torsion but that are not naturally reductive -- hence,
these cannot be classified.
\begin{thm}[case B.2 -- classification]\label{n=5EWgleich}
Let $(M^5,g,T)$ be a complete, simply connected  naturally reductive
homogeneous $5$-space such that $T = -(\vrho e_{125}+\lambda e_{345})$ 
with $\vrho\lambda \neq 0, \vrho\ = \lambda$. 
Then $M^5$ is either isometric to one of the spaces discussed in
Theorem $\ref{thm.B1-classification}$,
or to $\SU(3)/\SU(2)$ or $\SU(2,1)/\SU(2)$; in the last two
cases, the family of metrics depends on two parameters
(described in Section \ref{berger}).
\end{thm}
\begin{proof}
By assumption, 
 $\iso(T)=\un(2)$, hence the holonomy algebra $\hol^\nabla$ of $\nabla$
is a subalgebra of $\un(2)$. If $\hol^\nabla\subset \so(2)\oplus \so(2)$, we
are in the situation B.1 discussed previously. If not, it has to contain
$\su(2)$, i.\,e.~$\su(2)\subset \hol^\nabla \subset \un(2)$. 
The curvature operator is a symmetric, $\hol^\nabla$-invariant operator,
hence $\kr=\psi\circ \pi_{\un(2)}$ with a $\hol^\nabla$-equivariant map
$\psi:\ \un(2)\ra\un(2)$.
Since $\un(2)$ decomposes into $\su(2)\oplus \R\cdot\omega$ ($\omega$
the central element),  $\psi= a\, \Id_{\su(2)}\oplus b\,\Id_{\R}$ for
some real constants $a, b$, which in an explicit choice of basis
means (the squares denote symmetric tensor powers)
\be\label{kr.n=5EWgleich}
\kr\ =\ a [(e_{13}+ e_{24})^2 +
(e_{14} - e_{23})^2 + 
(e_{12}- e_{34})^2]
+ b \, (e_{12}+ e_{34})^2.
\ee
The  first Bianchi identity in the Clifford version 
(Theorem \ref{Bianchi-Clifford}) yields then $b - 3 a =\vrho^2$.
A routine computation shows then that $M$ is isometric to $\SU(3)/\SU(2)$
or its non-compact dual $\SU(2,1)/\SU(2)$.
\end{proof}
\begin{NB}
The classification of $5$-dimensional naturally reductive homogeneous spaces  
was obtained by O. Kowalski and L. 
Vanhecke in 1985 using other methods, see \cite{KV85}.
\end{NB}
Let us look from another point of view at $5$-dimensional Riemannian 
manifolds with parallel skew torsion and $\sigma_T \not= 0$.
The torsion induces
a canonical metric almost contact structure, i.\,e.
a (1,1)-tensor $\vphi:TM^5 \ra TM^5$, and $1$-form $\eta$ with dual vector 
field $\xi$ of length one such that
\be\label{contact-condition}
\vphi^2=-\Id+\eta\otimes\xi~~~~~ \mbox{ and } ~~~~~g(\vphi V,\vphi W)
=g(V,W)-\eta(V)\eta(W).
\ee
The fundamental form is then defined by $F(X,Y):= g(X,\vphi Y)$.
Moreover, a Nijenhuis tensor $N$ is defined by an expression very
similar to the one known from almost complex structures.  
A manifold with a metric almost contact structure is 
called 
\begin{enumerate}
\item
a \emph{quasi-Sasakian manifold} if $N = 0$  and $dF = 0$,
\item
an \emph{$\alpha$-Sasakian manifold} if $N = 0$  and $d\eta = \alpha F$.
The Sasaki case corresponds to $\alpha=2$. 
\end{enumerate}
%
%
\begin{thm}[case B -- induced contact structure]\label{thm-B-contact}
Let $(M^5,g,T)$ be an orientable  
Riemannian $5$-manifold with parallel skew torsion $T$ such that 
$\sigma_T\neq 0$. Then $M$ is a quasi-Sasakian manifold and $\nabla$ is its
characteristic connection. The quasi-Sasakian structure is $\alpha$-Sasakian
if and only if $\lambda=\vrho$ (case B.$2$), and it is 
Sasakian if $\lambda=\vrho=2$.
\end{thm}
\begin{proof}
The vector field $V:=*\sigma_T\neq 0$ is $\nabla$-parallel, and therefore 
Killing and of constant length. It satisfies $V\haken \sigma_T=0$
because $V\haken (*V)=0$. It defines a $\nabla$-parallel Killing vector 
field $\xi$ of unit length, 
\bdm
\xi\ =\ \frac{1}{\|\sigma_T\|}\,*\sigma_T\  =\  
\frac{1}{\vrho\lambda}\, *\sigma_T \ =\ e_5.
\edm
Denote by $\eta$ its dual $1$-form; in general, we identify
vectors and $1$-forms. Clearly $\nabla \eta=0$ as well,
and the formulas for computing the differential from the covariant derivative
yield then
\bdm
d\eta(X,Y)\ =\ \eta(T(X,Y))\ =\ g(\xi, T(X,Y))\ =\  T(X,Y,\xi),
\edm
hence, in the normal form of $T$,
$d\eta= de_5 = - ( \vrho e_{12} + \lambda e_{34})$.
In particular, both expressions together show that $T=\eta\wedge d\eta$.
Furthermore, $\nabla\xi=0$ is equivalent to $\nabla^g_X\xi=-T(X,\xi)/2$,
hence
\bdm
\nabla^g\xi\ =\ -(\frac{\vrho}{2}e_{12}+\frac{\lambda}{2}e_{34}).
\edm
We define an endomorphism $\vphi$ and its corresponding $2$-form 
$F(X,Y)=g(X,\vphi Y)$ by
\bdm
F\ =\ -(e_{12}+ e_{34}),\quad \vphi\ =\ e_{12} + e_{34}.
\edm
Let us remark that $F$ and $\vphi$ are globally well-defined.
If  $\vrho\neq\lambda$, $\m^2_1=\langle e_1, e_2\rangle$,
$\m^2_2=\langle e_3, e_4\rangle$ are $2$-dimensional
$\nabla$-parallel distributions,  the $2$-forms $\Omega_1:=e_{12}$ and 
$\Omega_2 :=e_{34}$
are $\nabla$-parallel and globally defined.
If
$\vrho=\lambda$, $F$ is proportional to $d\eta$, and again well-defined.
For later, we observe that $\nabla^g\xi=-\vphi$ if and only if 
$\vrho=\lambda=2$, and this is equivalent to $2F=d\eta$. One checks that 
$\vphi$, $\eta$ and $\xi$ satisfy  conditions (\ref{contact-condition}), 
which means that $(\xi,\eta,\vphi)$ defines a metric almost contact structure.
We now prove that this structure is quasi-Sasakian. First, consider $F$. 
For  $\vrho=\lambda$, we have $\vrho dF=dd\eta=0$.
Furthermore, $F$ is proportional to $\xi\haken T$, and this is a parallel
$2$-form, so $\nabla F=0$ as well, which is equivalent to $\nabla \vphi=0$.
If  $\vrho\neq\lambda$, we just observed that
$\Omega_1$ and $\Omega_2$ are $\nabla$-parallel, hence $\nabla F=0$, 
and we can conclude
with formula (\ref{par-2-form}) that $d \Omega_1=d \Omega_2=0$,
so $dF=0$. To sum up, we obtain 
$\nabla F=0$  and $dF=0$ for all coefficients $\vrho\lambda\neq 0$. 
This is the first of the
two conditions for a quasi-Sasakian structure. If a metric almost contact 
structure admits a connection with skew symmetric torsion, 
then $dF = 0$ implies 
that the Nijenhuis tensor vanishes, $N = 0$, see  \cite[Theorem 8.4]{FI02}. 
Thus, the structure is quasi-Sasakian as claimed.
\end{proof}
%
\section{The classification in dimension $6$}
%
A naturally reductive  $6$-space whose
torsion has non-trivial kernel is locally a product of lower-dimensional
spaces of the same kind by Theorem \ref{kerT-algebra-gT}. 
This applies, for example,
to a torsion that has the normal form (\ref{dim5-torsion}) and that depends
on a $5$-dimensional subspace only. Consequently, we will
assume henceforth that $\ker T=0$. We will discuss the different
cases according to the
algebraic properties of $\sigma_T$. Observe that $*\sigma_T$ is a $2$-form
and can hence be identified with a skew-symmetric endomorphism.
Thus, $\rk (*\sigma_T)$ is either $0,2,4$, or $6$; in analogy to the
discussion in five dimensions, we will label these cases A, B, C, and D.

First, we treat the
degenerate case $\sigma_T=0$ (case A): this covers, for example,
torsions with local normal form $\mu\,e_{123}+\nu\,e_{456}$.
Our result shows that this is basically the only possibility:
\begin{thm}[case A -- with parallel skew torsion]
Let $(M^6,g,T)$  be a complete, simply connected Riemannian $6$-manifold
with parallel skew torsion $T$ such that $\sigma_T=0$ and $\ker T=0$.
Then $M^6$ splits into two $3$-dimensional manifolds with parallel
skew torsion, 
\bdm
(M^6,g,T)\ =\ (N^3_1,g_1,T_1)\x (N^3_2,g_2,T_2).
\edm
\end{thm}
\begin{proof} %
First, we  argue that the Lie algebra $\g_T$ cannot act
irreducibly on  the tangent space $V=T_pM^6$. If so, it would define an 
irreducible skew holonomy system, so $\g_T=\so(6)$ by Theorem
\ref{thm.SHS}, since there are no $6$-dimensional compact simple Lie algebras.
But equation (\ref{eq.XT-auf-T}) and the condition  $\sigma_T=0$
imply
\bdm
\g_T\ =\ \so(6)\ \subset \ \iso (T), 
\edm
hence $T=0$, a contradiction. The assumption  $\ker T=0$ yields 
that $V$ can only split into two $3$-dimensional $\g_T$-invariant
subspaces. Now the claim follows from the splitting
Theorem  \ref{cor.splitting-g-T}.
\end{proof}
\begin{cor}[case A -- the naturally reductive case]
Any $6$-dimensional simply connected naturally reductive space with 
$\sigma_T=0$ and $\ker T=0$ is  isometric to a product of 
two $3$-dimensional  naturally reductive spaces. 
\end{cor}
Next we classify the case where $\ker T=0$ and the $2$-form $*\sigma_T$ 
has rank 2, 
\bdm
*\sigma_T \ = \ \rho \, e_{56} , \quad \sigma_T \ = \ \rho \, e_{1234}  .
\edm
The condition $\nabla T = 0$ implies $\nabla \sigma_T = 0$, 
i.\,e.~$\hol^\nabla  \subset   \iso(T)  \subset \so(4) \oplus \so(2)$.
We use again  equation (\ref{eq.XT-auf-T}), 
\bdm 
(e_5 \haken T)(T) \ = \ e_5 \haken \sigma_T \ = \ 0 \ = \  e_6 \haken \sigma_T \ = \ 
 (e_6 \haken T)(T)
\edm
and conclude that the $2$-forms $e_5 \haken T , e_6 \haken T$ belong to 
the Lie algebra $\iso(T) \subset \so(4) \oplus \so(2)$. Consequently, $T$ does 
not contain a term of type $v \wedge e_5 \wedge e_6 \, , \, v \in \R^4$. Let us 
write the $3$-form $T$ as
\bdm
T \ = \ T_1 \, + \, \Omega_1 \wedge e_5 \, + \, \Omega_2 \wedge e_6
\edm
where $T_1\in\Lambda^3(\R^4)$ and $\Omega_1,\Omega_2 \in\Lambda^2(\R^4)$ are forms
in $\R^4$. Then
$e_5 \haken T  =  \Omega_1$ and $e_6 \haken T = \Omega_2$,
i.\,e.~$\Omega_1 , \Omega_2 \in \iso(T)$ are linearly independent $2$-forms 
preserving $T$. In particular, they preserve $T_1$ and commute,
\bdm
\Omega_1 \wedge \Omega_2 \ = \ \Omega_2 \wedge \Omega_1  .
\edm  
Denote by $\so(2) \oplus \so(2) \subset \so(4)$ the maximal abelian subalgebra 
of $\so(4)$
generated by $\Omega_1$ and $\Omega_2$. The form $T_1 \in \Lambda^3(\R^4) \cong \R^4$ is 
invariant under $\so(2) \oplus \so(2)$, but this subalgebra has no fixed 
vectors in $\R^4$. We conclude that $T_1 $ vanishes. Moreover, 
we obtain
\bdm
\hol^\nabla \ \subset\ \iso(T) \ = \ \so(2)\oplus\so(2)\ \subset\ \so(4)\ 
\subset \ \so(6).
\edm
Up to a conjugation in $\so(4)$ we may assume that
\bdm
\Omega_1 \ = \ \alpha \, ( e_{12}  +  e_{34})  , \quad 
\Omega_2 \ = \ \beta \, ( e_{12}  - e_{34}) 
\edm
and 
\be\label{B-formulaT}
T \ = \ \alpha \, ( e_{12}  +  e_{34}) \wedge e_5  + 
 \beta \, ( e_{12}  -  e_{34}) \wedge e_6  .
\ee
We compute the square $\sigma_T$ directly,
\bdm
\sigma_T \ = \ (\alpha^2 \, - \, \beta^2) \, e_{1234} \ = \rho \, e_{1234} .
\edm
In particular, the length $\rho$ of $\sigma_T$ is given by $\alpha$ and 
$\beta$. The curvature operator $\kr$ is the composition of the projection 
onto $\so(2) \oplus \so(2)$ with an invariant linear map. Since 
$\so(2) \oplus \so(2)$ is abelian, $\kr$ depends on three 
parameters,
\be\label{B-formulaR}
\kr\ =\ \frac{a}{\alpha^2}\, \Omega_1\odot \Omega_1  + 
\frac{2 \, c}{\alpha \beta}\, \Omega_1\odot \Omega_2 + 
\frac{b}{\beta^2}\, \Omega_2\odot\Omega_2.
\ee
The Bianchi identity, i.\,e.~that 
$T^2 + \kr$ should be a scalar in the Clifford 
algebra, yields the relation
\bdm
\alpha^2   -  \beta^2 - a + b \ = \ 0 \ .
\edm
Finally, the curvature argument from Proposition
\ref{prop.sigma1}   yields the following result:
\begin{thm}[case B -- with parallel skew torsion]
Let $(M^6, g, T)$ be a complete, simply connected 
Riemannian $6$-manifold with parallel skew torsion and such that 
$\ker T=0$ and $\rk(*\sigma_T) = 2$. 
Then $\nabla \kr =0$, i.\,e.~$M$ is naturally reductive,
and the family of admissible torsion forms and curvature operators
depends on 5 parameters $\alpha, 
\beta,a,b,c$ satisfying the equation $\alpha^2 -\beta^2 - a + b = 0$.  
The torsion and the curvature  are given by the explicit formulas
$(\ref{B-formulaT})$ and $(\ref{B-formulaR})$. 
\end{thm}
In this situation, too, we describe explicitly the naturally reductive 
space $M^6$ and the $8$-dimensional group $G$ acting
transitively on it. Its Lie algebra $\g$ is an $8$-dimensional 
vector space with  basis $e_1, \ldots, e_6$ and 
\bdm
\tilde\Omega_1 \ := \ - \Big[ \frac{a + c}{\alpha} \, \Omega_1 + 
\frac{b + c}{\beta} \, \Omega_2  + 
\alpha \, e_5  +  \beta \, e_6 \Big],\quad
\tilde\Omega_2 \ := \ - \Big[ \frac{a - c}{\alpha} \, \Omega_1  + 
\frac{c - b}{\beta} \, \Omega_2 +\alpha \, e_5  -  \beta \, e_6 \Big] 
\edm 
We compute all commutator relations and observe that
the linear space $\g_1$ spanned by $e_1, e_2, e_3,
e_4$ and $\tilde\Omega_1, \tilde\Omega_2$ is a Lie subalgebra such that $\g_1 \cap \Lin(\Omega_1, \Omega_2) = 0$. 
Consequently, the simply connected manifold $M^6$ is isometric to the $6$-dimensional Lie group
$G_1$ defined by $\g_1$ and equipped with a family of left-invariant metrics.
The non-vanishing commutator relations in $\g_1$ are
\begin{gather*}
[e_1,e_2]\, =\,\tilde\Omega_1,\quad
[\tilde\Omega_1,e_1]\, =\, - (a + b + 2c - \alpha^2 - \beta^2)\, e_2,\quad
[\tilde\Omega_1 , e_2]\, =\, (a + b + 2c - \alpha^2 - \beta^2)\, e_1,\\ 
[e_3,e_4]\, =\,\tilde\Omega_2,\quad
[\tilde\Omega_2,e_3]\, =\, - (a + b - 2c - \alpha^2 - \beta^2)\, e_4,\quad
[\tilde\Omega_2, e_4]\, =\,  (a + b - 2c - \alpha^2 - \beta^2)\, e_3.
\end{gather*}
The Lie algebra $\g_1$ splits into two $3$-dimensional ideals and 
any of them is isomorphic
to $\so(3)$, $\sl(2,\R)$ or to the Heisenberg algebra $\h^3$. 
We have thus described completely the naturally 
reductive spaces classified in the previous Theorem. Remark that in 
the flat case ($a = b = c = 0$)
we obtain the result of Cartan and Schouten, see \cite{Agricola&F10},
whereby a $\nabla$-flat space with skew symmetric torsion 
is a product of compact simple Lie groups and the $7$-dimensional sphere.
\begin{thm}[case B -- classification]\label{thm.B-classification}
A complete, simply connected Riemannian $6$-manifold with parallel 
skew torsion $T$ and $\rk(*\sigma_{T}) = 2$ 
is the product $G_1\x G_2$ of two Lie groups equipped with a family of 
left-invariant metrics.
$G_1$ and $G_2$ are either $S^3=\SU(2)$,  $\widetilde{\SL}(2,\R)$, or $H^3$.
\end{thm}
\begin{NB}
An explicit description of naturally reductive structures with 
$\rk(*\sigma_{T}) = 2$ on $S^3\x S^3$ may be found in \cite[Exa 4.8]{Sch07}.
They are disjoint to the examples with $\rk(*\sigma_{T}) = 6$ that will be
described in Section \ref{S3xS3}.
\end{NB}
\begin{exa}
There are $3$-forms $T \in \Lambda^3(\R^6)$ such that the corresponding
$2$-form $* \, \sigma_T \in \Lambda^2(\R^6)$ has rank $4$. For example, consider
the family
\bea[*]
T &=& e_5 \wedge (a e_{12}  +  b e_{13}  +  c e_{14}  +  
d e_{23}  +  f e_{24}  +  h e_{34}) \\ 
& + & e_6 \wedge (s e_{12}  +  t e_{13}  +  u e_{14} +  
v e_{23}  +  w e_{24}  +  x e_{34})      
\eea[*]
depending on $12$ parameters. It turns out that the vectors $e_5 , e_6$ 
are in the  kernel of $* \, \sigma_T$ if and only if the equation
\bdm
cd  -  bf \, + \, ah \, + \, uv \, - \, tw \, + \, sx \ = \ 0
\edm
is satisfied. Consequently, a generic choice of parameters constrained 
by the latter equation yields $3$-forms with $\rk(*\sigma_T) = 4$. 
For example, the parameter values
$a = b= c=d= h= -f =1$, $s  =  \frac{1}{2}$ ,  $t= u= w = -v=1$ , $x =-2$  
are one out of many solutions with $\rk(*\sigma_T) = 4$.
Surprisingly, our next result states that such $3$-forms cannot
occur as torsions  forms of metric connections with parallel skew torsion,
thus showing that $\nabla T=0$ is a severe restriction on the algebraic
type of a $3$-form.
\end{exa}
Almost Hermitian $6$-manifolds have structure group $\U(3)\subset \SO(6)$.
If we decompose $\so(6)=\mathfrak{u}(3)\oplus \m$, they are classified by
their intrinsic torsion, which is an element of $\m\ox\R^6$. 
This was first done by Gray and Hervella in \cite{GH80}, who introduced
the by-now standard notation
\bdm
\m\ox\R^6\ =\ \mathcal{W}_1^{(2)}\oplus \mathcal{W}_2^{(16)}\oplus
 \mathcal{W}_3^{(12)}\oplus \mathcal{W}_4^{(6)} 
\edm
for its decomposition into $\U(3)$-modules (the upper index indicates the
real dimension and will be dropped later; see also \cite{Alexandrov&F&S05}
for a modern account). An almost Hermitian $6$-manifold 
admits a characteristic connection if and only if the intrinsic torsion 
has no $\mathcal{W}_2$-component \cite{FI02}. More precisely,
$\Lambda^3(\R^6)=\mathcal{W}_1\oplus \mathcal{W}_3\oplus \mathcal{W}_4$ as a 
$\U(3)$-module,
 the torsion of the characteristic connection is a linear combination
$T=T_1+T_3+T_4$ with $T_i\in \mathcal{W}_i$ and the Gray-Hervella class
of the almost Hermitian structure corresponds to the non-vanishing
contributions of $T$. For example,
almost Hermitian manifolds of  class $\mathcal{W}_1$ go under the name
nearly K\"ahler manifolds, while Hermitian manifolds 
(vanishing Nijenhuis tensor) are of class $\mathcal{W}_3\oplus\mathcal{W}_4$.
%
\begin{thm}[cases C and D -- eigenvalues of $\sigma_T$]\label{caseCDn=6}
%
If $T$ is the torsion form of a metric connection $\nabla$ 
in dimension $6$ with parallel skew torsion, 
$\ker T = 0$ and $\rk(*\sigma_T) \geq 4$, then all eigenvalues 
of $*\sigma_T$ are
equal. The $2$-form defines a $\nabla$-parallel almost Hermitian 
structure $J$ of Gray-Hervella class  
$\mathcal{W}_1 \oplus \mathcal{W}_3$. In particular, the 
case $\rk(*\sigma_T) = 4$ cannot occur.
\end{thm}
\begin{proof}
If the eigenvalues of $*\sigma_T$ are equal, this $2$-form can 
be viewed, up to
a factor, as  the K\"ahler form of an almost Hermitian structure. 
The metric as well as the almost complex structure are preserved by 
a connection 
with totally skew symmetric torsion. This implies that the  
$\mathcal{W}_2$-part 
in the Gray-Hervella classification vanishes.  Moreover, 
$dT = 2 \, \sigma_T$ yields that $*\sigma_T$ is coclosed, $\delta(*\sigma_T) =
 d(\sigma_T) = 0$, i.e. the $\mathcal{W}_4$-part  
vanishes too (see \cite{GH80}, \cite{FI02} or \cite{Alexandrov&F&S05}). 

By contradiction, let use assume that $*\sigma_T$ has two different 
eigenvalues. 
At least one of them defines a $\nabla$-parallel $2$-dimensional 
subbundle $E^2$, i.e. the tangent 
bundle splits into $TM^6 = E^2 \oplus F^4$ and
 the holonomy algebra of $\nabla$ is contained in $\so(E^2) \oplus
\so(F^4)$. In case $\rk(*\sigma_T) = 4$,  we put $E^2 := \ker(*\sigma_T)$. Thus,
the $2$-form $*\sigma_T$ is non-degenerate on the subbundle $F^4$ in 
all cases. We fix a local frame $e_5, e_6$ in $E^2$. 
First we prove that the subbundle $F^4$ does not admit
non-trivial  
$\nabla$-parallel vector fields. In fact, if $V$ is one, then $V \haken 
(*\sigma_T)$ is $\nabla$-parallel too. This restricts the holonomy algebra,
\bdm
\hol^{\nabla} \ \subset \ \so(E^2) \, \oplus \, \mathfrak{u} ,
\edm
where $\mathfrak{u} \subset \so(F^4)$ is a $1$-dimensional subalgebra. 
We fix a basis $e_1, e_2, e_3 = V, e_4 = V \haken (*\sigma_T)$ 
in $F^4$ as well as a generator $\Omega$ of the subalgebra 
$\mathfrak{u} \subset \so(F^4)$. Since 
$e_3 \haken \Omega = e_4 \haken \Omega = 0$,
$\Omega$ is given by $\Omega = a \, e_{12}$. Consider the curvature operator of 
$\nabla$,
\bdm
\kr \ = \ \alpha \, e_{56} \odot e_{56} \, + \, \beta \,  e_{56}  \odot \Omega 
\, + \, \gamma \, \Omega \odot \Omega \ .
\edm
In the Clifford algebra the Bianchi identity reads as
$-  2 \, \sigma_T  +  \kr  =  T^2  +  \kr  =  0$
modulo scalars. This implies the following formula for $\sigma_T$,
\bdm
\sigma_T \ = \ \frac{\alpha \, \beta}{2} \,  e_{1256} \ ,
\edm
i.\,e.~$2 *\sigma_T = \alpha \beta \, e_{34}$ has rank $2$, a contradiction. 
This proves 
that $F^4$ has no $\nabla$-parallel vector fields.
We apply this fact to $e_{56} \haken T$ and to the restriction 
$T_{|F^4} \in \Lambda^3(F^4) \cong F^4$. They are
$\nabla$-parallel vector fields in $F^4$, hence they should vanish. 
Therefore, the torsion form can be written locally as
\bdm
T \ = \ e_5 \wedge \Omega_1 \ + \ e_6 \wedge \Omega_2
\edm
where $\Omega_1, \Omega_2 \in \so(F^4)$ are $2$-forms in $F^4$. Since the 
kernel of $T$ is trivial and 
$\Omega_1 = e_5 \haken T , \, \Omega_2 = e_6 \haken T$,
the forms $\Omega_1, \Omega_2$ are linearly independent. Using the 
well-known formulas for multiplication inside the Clifford algebra
\bdm
T^2 \ = \ -   2 \, \sigma_T  +  \|T\|^2  , \quad \Omega^2_i \ = 
\ - \|\Omega_i\|^2  +  \Omega_i \wedge \Omega_i  , \quad i = 1,2 
\edm
we can compute $\sigma_T$ :
\bdm
2 \, \sigma_T \ = \ ( \Omega_1 \wedge \Omega_1 +  
\Omega_2 \wedge \Omega_2 )  - \ [\Omega_1  ,  \Omega_2] \wedge e_{56} .  
\edm
The forms $\Omega_1, \Omega_2$ and $[\Omega_1,\Omega_2]$ are 
linearly independent, otherwise the rank of $\sigma_T$ would not 
be six or the Lie algebra $\so(F^4)$ would contain a non-abelian 
$2$-dimensional subalgebra. Since $\sigma_T$ and $e_{56}$ are $\nabla$-parallel,
the $2$-form $[\Omega_1 , \Omega_2]$ is $\nabla$-parallel too. The holonomy 
algebra $\hol^{\nabla}$ is at least one-dimensional 
(see Proposition \ref{prop.sigma1}) and preserves the $2$-dimensional spaces 
$\Lin(e_5,e_6)$ and $\Lin(\Omega_1,\Omega_2)$. Consider an arbitrary element 
$\Sigma = A \, e_{56}  + \Sigma_1 \in \so(E^2) \oplus \so(F^4)$. The 
element $\Sigma$ preserves $T$ if and only if
\bdm
\Sigma(e_5) \ = \ A \, e_6  , \quad 
\Sigma(e_6) \ = \ -  A \, e_5  , \quad
[\Sigma_1  ,  \Omega_1 ] \ = \ A \, \Omega_2 , \quad 
[\Sigma_1 , \Omega_2 ] \ = \ -  A \, \Omega_1 .
\edm
The holonomy algebra cannot lie completely inside $\so(F^4)$. Indeed,
the existence of a curvature operator $\kr : \Lambda^2 \rightarrow \hol^{\nabla}$
as well as the Bianchi identity implies in this case that
$2  \sigma_T \ = \ \alpha \, e_{1234}$,
i.\,e.~$*\sigma_T$ has rank $\leq 2$, a contradiction. Consequently, 
there exists at least one element
$\Sigma = A \, e_{56}  +  \Sigma_1 \in \hol^{\nabla}$ 
with
$A \not= 0$. The commutator relations  imply that 
the $2$-forms $\Omega_1, \Omega_2$ are orthogonal 
with equal lengths. Then $\Omega_1, \Omega_2$ generate a 
Lie algebra $\mathfrak{h}$ such that $\dim \mathfrak{h} \geq 3$.

If $\dim \mathfrak{h} =  3$, the vector space 
$\Lin(\Omega_1, \Omega_2, [\Omega_1,\Omega_2])$  
is a Lie subalgebra. Since the kernel of $T$ is trivial, the representation of 
this algebra in $\R^4$ is irreducible. Therefore this algebra can be conjugated 
into $\su(2) \cong \Lambda^2_{-}(F^4)$ or $\Lambda^2_{+}(F^4)$. We 
know already that $\Omega_1, \Omega_2$ are orthonormal and have the 
same length, hence (in an appropriate basis)
\bdm
\Omega_1 \ = \ e_{12}  -  e_{34}, \ 
\Omega_2 \ = \ e_{13}  + e_{24} , \quad 2 \, \sigma_T \ = \ 
-  4 \, e_{1234}  +  4 \, (e_{23} - e_{14}) \wedge e_{56}  ,
\edm
i.\,e.~the eigenvalues of $\sigma_T$ are equal, a contradiction. Thus, 
$\dim \h =  3$ cannot occur. In the next step we prove that the 
holonomy algebra is one-dimensional, $\dim \hol^{\nabla}=  1$.
Indeed, suppose that $\dim \hol^{\nabla} \geq  2$. Then there exists an 
element $0 \not= \Sigma^* \in \hol^{\nabla} \cap \so(F^4)$ and 
\bdm
\Sigma^*(e_5) \ = \ \Sigma^*(e_6) \ = \ 0  , \quad
[\Sigma^* , \Omega_1] \ = \ [\Sigma^* , \Omega_2] \ = \ 0 
\edm 
hold. Then $\Sigma^*$ commutes with the whole Lie algebra $\mathfrak{h}$ 
generated by
$\Omega_1 , \Omega_2$, $[\Sigma^* , \mathfrak{h}] = 0$. 
Since the dimension of $\h$ is at least $4$, 
the existence of $\Sigma^*$ implies automatically that
$\dim \h  =  4$, $\Sigma^*  \in  \h$.
Moreover, $\Sigma^*$ is the central element of $\h$.
Up to a conjugation inside $\so(F^4)$, we know the Lie algebra has to be
\bdm
\h \ = \ \Lambda^2_{-}(F^4) \, \oplus \, \R \cdot (\Sigma^*)  
\edm
and
\bdm
\Sigma^* \ = \ e_{12}  + e_{34} , \ 
\Omega_1 \ = \ \Omega^-_1  +  a \, \Sigma^* \, , \
\Omega_2 \ = \ \Omega^-_2  +  b \, \Sigma^* \, , \ \Omega^-_1 , \Omega^-_2 \, 
\in \, \Lambda^2_{-}(F^4)   
\edm
for some constants $a,b \in \R$.
We use once again the existing element with $A\not= 0$, namely 
$\Sigma = A  e_{56}  +  \Sigma_1  =  A \, e_{56} + \Sigma_1^{+}  
+ \Sigma_1^{-} \in\hol^{\nabla}$ with $ \Sigma_1^{\pm}\in\Lambda_{\pm}^2(F^4)$. 
Comparing the $\Lambda_{\pm}^2$-parts in $[\Sigma_1, \Omega_1] = A \, \Omega_2, 
[\Sigma_1, \Omega_2] = - A \, \Omega_1$ we obtain 
\bdm
A \, \Omega_2^{-} \ = \ [ \Sigma_1^{-}, \Omega_1^{-}] , \quad
A b \, \Sigma^* \ = \ a \, [\Sigma^+_1  ,  \Sigma^*],\quad 
- \,  A \, \Omega_1^{-} \ = \ [ \Sigma_1^{-}  ,  \Omega_2^{-}] , \quad
- \, A a \, \Sigma^* \ = \ b \, [\Sigma^+_1  ,  \Sigma^*].  
\edm
The elements $\Sigma^*$ and $[\Sigma^+_1 , \Sigma^*]$ are orthogonal in 
$\Lambda^2_{+}$. Since $A \not= 0$ we conclude that $a = b = 0$. Then 
$\Omega_1, \Omega_2 \in \Lambda^2_{-}$ and the algebra $\mathfrak{h}$ is 
$3$-dimensional; but we know already that this case is impossible.
Consequently, the element $0 \not= \Sigma^*$ cannot exist, 
i.\,e.~$\dim \hol^{\nabla} = 1$.  This algebra is generated by one element 
$\Sigma = A \, e_{56}  +  \Sigma_1 , A \not= 0, \ \Sigma_1 \in \Lambda^2(F^4)$.
Using the curvature operator as well as the Bianchi identity again, we obtain
relations between $\Sigma_1$ and $\Omega_1, \Omega_2$,
\bdm
2A \, \Sigma_1 \ = \ - \, [\Omega_1 \, , \, \Omega_2] \, , \ 
[\Sigma_1 \, , \, \Omega_1] \ = \ A \, \Omega_2 \, , \
[\Sigma_1 \, , \, \Omega_2] \ = \ - \, A \, \Omega_2 \, .
\edm
The vector space $\Lin(\Omega_1 , \Omega_2, [\Omega_1,\Omega_2])$ becomes
a Lie algebra and coincides with $\mathfrak{h}$, again a contradiction.
This finishes the proof. 
\end{proof}
\begin{NB}
Any $2$-form $*\sigma_T$ of rank $\geq 4$ defines (locally) in a natural way 
a $\nabla$-parallel almost Hermitian structure $J$ on $(M^6,g)$. N.~Schoemann 
classified in
the paper \cite{Sch07} all admissible torsion forms of almost Hermitian
structures with characteristic torsion, up to  conjugation 
under $\U(3)$. Applying these explicit formulas for $T$, one 
can compute $*\sigma_T $ in all cases. This yields an
alternative, more computational proof of 
Theorem \ref{caseCDn=6}.
\end{NB}
Consider a complete, simply connected Riemannian $6$-manifold $(M^6, g, T)$ 
with parallel
skew torsion $T$, $\rk(* \sigma_T)=6$ and $\ker T = 0$. If the holonomy 
representation
is $\C$-reducible, we obtain again a $\nabla$-parallel decomposition 
$T(M^6) = E^2 \oplus F^4$ 
and the torsion form is given by $T = e_5 \wedge \Omega_1 
+ e_6 \wedge \Omega_2$. The proof of Theorem
 \ref{caseCDn=6} says that the linear space 
$\h:=\Lin(\Omega_1, \Omega_2, [\Omega_1,\Omega_2])$ is
a $3$-dimensional Lie subalgebra of $\so(F^4)$ and that it can be 
conjugated into $\Lambda^2_{\pm}(F^4)$. 
The two possible cases $\h=\Lambda^2_{\pm}(F^4)$ yield torsion 
forms of pure type $\mathcal{W}_1$ or $\mathcal{W}_3$, respectively. 
We conclude that $(M^6,g)$ is 
isometric to $\CP^3$ or $F(1,2)$, the twistor spaces  of the 
$4$-dimensional spaces $S^4$ or $\CP^2$ 
equipped with their nearly K\"ahler metric, see 
\cite{BM01} and \cite[Thm 4.5]{Alexandrov&F&S05} (the two different 
almost complex structures correspond to
different orientations of $J$ in the fibre direction). If the holonomy  
representation
is $\C$-irreducible, there are four possibilities: 
\bdm
\Hol^{\nabla} \ = \  \U(3) , \ \SU(3) , \ \SU(2)/\{\pm 1\} = \SO(3) 
\subset \SU(3), \ \U(2)/\{\pm 1\}  \subset  \U(3).
\edm
The first case gives $T=0$. 
If $\hol^{\nabla} = \su(3)$,  the almost Hermitian structure is of pure 
type $\mathcal{W}_1$ and 
$(M^6,g)$ is isometric to a nearly K\"ahler $6$-manifold with irreducible 
holonomy of $\nabla$. 
The discussion of the other cases provides the following final result:
\begin{thm}[case D -- with parallel skew torsion]\label{Dn=6, classification}
Let $(M^6, g, T)$ be a complete, simply connected Riemannian $6$-manifold 
with parallel
skew torsion $T$, $\rk(* \sigma_T)=6$ and $\ker T = 0$. Then one of the 
following cases occurs:
\begin{enumerate}
\item[] Case $D.1$: $(M^6,g)$ is isometric to a nearly K\"ahler $6$-manifold.
\item[] Case $D.2$: $\hol^{\nabla} = \so(3) \subset \su(3)$ and $(M^6,g)$ is 
naturally reductive. The 
family depends on three parameters
and is of type $\mathcal{W}_1 \oplus \mathcal{W}_3$ in the 
Gray-Hervella classification of 
almost Hermitian structures. 
\end{enumerate}  
Furthermore, $\hol^{\nabla}$ cannot be equal to 
$\mathfrak{u}(2)\subset \mathfrak{u}(3)$ with is $\C$-irreducible representation
on $\C^3$.
\end{thm}  
\begin{proof}
Let us fix a Lie subgroup $\SO(3)\subset \SU(3)\subset \SO(6)$;
for computations, we choose the Lie subalgebra $\so(3)$ given by
\bdm
H_1\, :=\, -2(e_{35}+e_{46}),\quad 
H_3\, :=\, 2(e_{15}+e_{26}),\quad 
H_5\, :=\, -2(e_{13}+e_{24})
\edm 
with commutators $[H_1,H_3]=-2H_5,\ [H_1, H_5]= 2 H_3,\ [H_3,H_5]=-2H_1$.
The group $\SU(3)$ acts trivially on $\mathcal{W}_1 \subset \Lambda^3(\R^6)$,
which is spanned by the two elements 
%
\bdm
 - e_{135} + e_{245} + e_{236} + e_{146}\ =:\ - e_{135} + \eta_1 ,\quad
 - e_{246} + e_{136} + e_{145} + e_{235}\ =:\ - e_{246} + \eta_2.
\edm
The space of $\SO(3)$-invariant $3$-forms in 
$\mathcal{W}_3 \subset \Lambda^3(\R^6)$ has complex dimension one,
 and, for our choice of $\so(3)$, it is spanned by
$ 3\,e_{135} + \eta_1$ and $3\, e_{246}+\eta_2$ 
\cite[p.7 and Thm 4.4]{Alexandrov&F&S05}.
Hence  the general
 $\SO(3)$-invariant $3$-form in $\mathcal{W}_1 \oplus \mathcal{W}_3$
lies in $\Lin(e_{135}, e_{246},\eta_1,\eta_2)$ and 
depends a priori on $4$ real parameters.  The one-dimensional center
of $\U(3)$ leaves  the Lie algebra $\so(3)\subset \su(3)$ invariant and
its  generator acts on $\Lin(e_{135}, e_{246},\eta_1,\eta_2)$ by transforming 
$e_{135}$ into $e_{246}$
and $\eta_1$ into $\eta_2$. Consequently, it is possible
to conjugate  by a central element of $\U(3)$ in such a way that
$\eta_2$ disappears. 
With this choice,  the general $\SO(3)$-invariant $3$-form in 
$\mathcal{W}_1 \oplus \mathcal{W}_3$ can be parametrized as
\be\label{n=6-torsion}
T\ =\ \alpha\, e_{135}+ \alpha'\, e_{246} + \beta\, (e_{245}+ e_{236}+e_{146}),
\ \  \text{ hence }
\sigma_T\ =\ \beta (\beta-\alpha)\, (e_{1256}+ e_{1234}+e_{3456}).
\ee
We see that we need to require $\beta\neq 0$ and $\alpha\neq \beta$
to ensure that $\rk (*\sigma_T)= 6$. 
The curvature operator $\kr : \Lambda^2(\R^6) \rightarrow \so(3)$
depends a priori on nine parameters $x_{ij}$,
\bdm
\kr \ = \ \sum_{i,j=1}^{3} x_{ij} \, H_i  \odot H_j \, .
\edm
The Bianchi identity $2 \sigma_T = \kr$ modulo scalars in the Clifford 
algebra constrains the coefficients, namely $x_{11}=x_{22}=x_{33}$
and $x_{ij} = 0$ for $i \not= j$. In particular, the curvature operator 
is the projection onto
the subalgebra $\so(3)$. Consequently, $\kr$ is $\hol^{\nabla}$-invariant 
and the space is
naturally reductive. The explicit formula for the curvature is
\be\label{n=6-curvature}
\kr \ = \ \beta(\alpha-\beta) [(e_{35}+e_{46})^2+ (e_{15}+e_{26})^2 + 
(e_{13}+e_{24})^2].
\ee
The torsion and the curvature depend on three parameters and they describe the 
naturally reductive spaces completely. The underlying almost Hermitian 
manifold is of pure type $\mathcal{W}_1$ or $\mathcal{W}_3$ if 
$\alpha+\beta=0, \,\alpha'=0$ or resp.~$ \alpha=3\beta, \,\alpha'=0$.
In the next paragraph we will describe these spaces
as well as the transitive automorphism groups in a more explicit way.
Finally, the case of the irreducible representation 
$\Hol^{\nabla} = \U(2)/\{\pm 1\} \subset 
\U(3)$ cannot occur. Indeed, a $\mathcal{W}_1$-component of the torsion
is then impossible, and there is no  
admissible torsion in $\mathcal{W}_3\subset\Lambda^3(\R^6)$ either, see
\cite[Thm 4.4]{Alexandrov&F&S05}.
\end{proof}
\begin{NB}\label{NB-nK}
Homogeneous nearly K\"ahler $6$-manifolds have been classified in 
\cite{Butruille05}. 
The list is short: $S^6 , S^3 \times S^3, \CP^3$, and the flag manifold 
$F(1,2)=\U(3)/\U(1)^3$, all equipped with their unique strict nearly K\"ahler 
metrics. It is well known that they are  naturally reductive,
and that they are precisely the Riemannian $3$-symmetric spaces. 
Using the result of Butruille and the previous Theorem,  we immediately 
obtain the
classification of all naturally reductive spaces in case D.1.
\end{NB}
\begin{NB}\label{NB-curvature}
All naturally reductive spaces described in Theorem
\ref{Dn=6, classification} are Einstein with parallel skew torsion
(see \cite{Agricola&F12} for details on this notion). Indeed, the
expression (\ref{n=6-curvature}) for the curvature implies that the Ricci
tensor of the connection $\nabla$ is given by
\bdm
\Ric = -2\beta(\alpha- \beta)\,\Id .
\edm
The difference $\Ric-\Ric^g$ depends only on the torsion tensor;
one checks that it is a multiple of the identity if and only if
$\alpha'=0$ and $\alpha=\pm \beta$. Hence, these are the only situations
where the space will be Riemannian Einstein, and these are
well known: $\alpha'=0,\alpha= \beta $ is the bi-invariant metric on
$S^3\x S^3$ (recall that $\kr=0$ in this case), 
while $\alpha'=0,\alpha= -\beta $ corresponds to nearly
K\"ahler manifolds.
\end{NB}
We shall now carry out the Nomizu construction (see Appendix \ref{app.nomizu})
for the naturally
reductive  spaces covered by case D.2 of the previous theorem;
we keep the Ansatz for the torsion (equation (\ref{n=6-torsion})) and 
the notation of the proof.  The only non-vanishing commutators on
$\Lin (H_1,H_3,H_5) \oplus \Lin(e_1,\ldots, e_6) $   are then (besides
the commutators of the elements $H_i$ already mentioned in the above proof) 
\begin{align*}
[e_1,e_3] = \frac{\beta(\alpha-\beta)}{2}H_5 - \alpha e_5, & \ \ \
[e_2,e_4] = \frac{\beta(\alpha-\beta)}{2}H_5 - \beta e_5 -\alpha' e_6, &  
[e_1,e_4] = [e_2,e_3] = -\beta e_6,\\
[e_1,e_5] = \frac{\beta(\beta-\alpha)}{2}H_3 + \alpha e_3, & \ \ \
[e_2,e_6] = \frac{\beta(\beta-\alpha)}{2}H_3 + \beta e_3+\alpha' e_4, & \ \ 
[e_1,e_6] = [e_2,e_5] = +\beta e_4,\\
[e_3,e_5] = \frac{\beta(\alpha-\beta)}{2}H_1 - \alpha e_1, & \ \ \
[e_4,e_6] = \frac{\beta(\alpha-\beta)}{2}H_1 - \beta e_1 -\alpha' e_2, &  
[e_3,e_6] = [e_4,e_5] = -\beta e_2,
\end{align*}
as well as
\begin{align*}
[e_5,H_3]=[H_5,e_3]=2 e_1, & \ \ \quad [e_1,H_5]=[H_1,e_5]=2e_3, 
& [e_3,H_1]=[H_3,e_1]=2e_5,\\
[e_6,H_3]=[H_5,e_4]=2e_2, & \ \ \quad [e_2, H_5]=[H_1, e_6]=2 e_4, 
& [e_4, H_1]=[H_3,e_2]=2e_6.
\end{align*}
In the next step, we  find a $6$-dimensional Lie subalgebra. 
Define $\Omega_i:= e_i+\frac{\beta-\alpha}{2}H_i$ for $i=1,3,5$,
$\h:= \Lin(\Omega_1, \Omega_3,\Omega_5)$, and
$\m:=\Lin(e_2,e_4,e_6)$. Then $\g:=\h\oplus\m$ is indeed $6$-dimensional
and, as the following non-vanishing commutators show, closed under
the Lie bracket. More precisely, $\h$ is a $3$-dimensional Lie subalgebra, 
\bdm
[\Omega_1,\Omega_3]\, =\,  (\alpha-2\beta)\Omega_5,\quad
[\Omega_1,\Omega_5]\, =\,  (2\beta-\alpha)\Omega_3,\quad
 [\Omega_3,\Omega_5]\, =\,  (\alpha-2\beta)\Omega_1,\quad
\edm
that is either abelian ($\alpha=2\beta$) or isomorphic to $\so(3)$ 
($\alpha\neq 2\beta$). The space $\m$ is a reductive complement of $\h$
inside $\g$,
\bdm
[\Omega_1, e_4] =  [e_2, \Omega_3] = (\alpha-2\beta)e_6, \ 
[\Omega_1, e_6] =  [e_2,\Omega_5] = (2\beta-\alpha) e_4, \ 
[\Omega_3, e_6] = [e_4,\Omega_5] = (\alpha - 2\beta) e_2.
\edm
The remaining commutators of elements from $\m$ are
\bdm
[e_2, e_4] = -\beta\Omega_5-\alpha' e_6,\quad
[e_2,e_6] = \beta\Omega_3 + \alpha' e_4, \quad
[e_4, e_6] = -\beta \Omega_1-\alpha' e_2.
\edm
Since $\g\cap\Lin (H_1,H_3,H_5)=0$, we can conclude that $M^6$ is 
isometric to the $6$-dimensional simply connected Lie group $G$ with
Lie algebra $\g$ equipped with 
a left-invariant metric. To determine $G$, we compute
the Killing form $\tilde\beta(x,y):=\tr (\ad x \circ\ad y)$
in the ordered basis $\Omega_1, \Omega_3,\Omega_5,e_2,e_4,e_6 $ of $\g$,
\bdm
\tilde\beta\ =\  
 \left[ \begin{array}{@{}c|c@{}}
-4(\alpha-2\beta)^2 \cdot 1_3     
& 2 \alpha' (\alpha-2\beta)\cdot 1_3   \\ \hline
2 \alpha' (\alpha-2\beta)\cdot 1_3      
& \big(4\beta(\alpha-2\beta)-2 {\alpha'}^2 \big)\cdot 1_3
    \end{array} \right]\, .
\edm
The matrix $\tilde\beta$ always has two eigenvalues of multiplicity $3$
each, hence $\rk \tilde\beta= 0 , 3,$ or $6$.
We can immediatly identify $G$ if $\rk \tilde\beta= 6$, for then
it has to be a semisimple Lie algebra with either negative
definite or split Killing form. These are exactly
$\so(3)\oplus \so(3)$ or $\so(3,1)\cong\sl(2,\C)$. Later,
we will give an alternative argument for this result. 
Let us look at
the determinant  to identify the cases  when
$\tilde\beta$ does not have full rank, i.\,e.~when $G$ is not
semisimple,
\bdm
\det \tilde\beta\ =\  -64 (\alpha-2\beta)^6 (4\beta(\alpha-2\beta) 
- {\alpha'}^2)^3.
\edm
Therefore, there are  three singular cases to be discussed 
separately: $\alpha-2\beta=0$,   $4\beta(\alpha-2\beta) - {\alpha'}^2=0$,
or both conditions simultaneously.
Let us first treat the last case, i.\,e.~$\alpha=2\beta$ and $\alpha'=0$. 
Then $\g$ is a 2-step nilpotent Lie algebra. 
With respect to the new basis $a_1=e_2, a_2=e_4, a_3=e_6, a_4=\beta\Omega_5,
a_5=-\beta\Omega_3, a_6=\beta\Omega_1$, it can alternatively be described
by
\bdm
da_1\, =\, da_2\, =\, da_3\, =\, 0,\quad da_4\, =\, a_{12}, \quad
da_5\, =\, a_{13}, \quad da_6\, =\, a_{23}.
\edm
Thus, it corresponds in standard notation to the nilpotent Lie algebra
$(0,0,0,12,13,23)$. Alternatively, it can be described as $\R^3\x\R^3$
with the commutator
\bdm
[(v_1, w_1), (v_2, w_2)]\ =\ (0, v_1\x v_2).
\edm
In this realisation, it may be found as an example in Schoemanns' work
\cite[Example 4.13, p.2208]{Sch07}. Up to a constant, the metric is 
unique.

Let us now consider one of the two cases where $\rk \tilde\beta=3$,
namely, $\alpha=2\beta$, but $\alpha'\neq 0$. Then three blocks of
$\tilde\beta$ vanish, and the lower right $(3\x 3)$ block 
is $-2{\alpha'}^2\cdot 1_3$. In the new basis
$\Omega_1,\Omega_3, \Omega_5, f_2:= \beta\Omega_1+\alpha'e_2,
f_4=\beta\Omega_3+\alpha'e_4, f_6:= \beta\Omega_5+\alpha'e_6,$
the only non-vanishing Lie brackets are
\bdm
[f_2,f_4]\ =\ -{\alpha'}^2f_6,\quad
[f_2,f_6]\ =\ {\alpha'}^2f_4,\quad
[f_4,f_6]\ =\ -{\alpha'}^2f_2.
\edm
We conclude that $\g=\R^3\x \so(3)$, hence $G$ is the direct
product $\R^3\x S^3$.

We shall now discuss the case $\alpha\neq 2\beta$; as it turns
out, the value of $\alpha'$ will only matter at the very end.
Instead of looking at $\g=\h\oplus \m$ itself, we look at the
homogeneous space it defines: we observed before that $\m$
is a reductive complement of the Lie subalgebra $\h$, and
by the assumption $\alpha\neq 2\beta$, the representation of
$\h$ on $\m$ is irreducible and $\h$ is isomorphic to $\so(3)$.
Hence, $\h$ defines a \emph{compact} (in particular, closed) subgroup of 
$G$, and $G/H=:P^3$ is a  $3$-dimensional manifold, of course homogeneous. 
Recall that the action of $G$ on $G/H$ is effective if and only if
$\h$ contains no nontrivial ideal of $\g$. This is obviously
the case if $\alpha\neq 2\beta$, making $G$ a subgroup 
of the isometry group of
$P^3$. However, the isometry group of a $3$-dimensional manifold is
at most $6$-dimensional. We obtain the remarkable
result that $ \widetilde{\mathrm{Iso}(P^3)} = G$ has maximal dimension.
Then
$P^3$ has to be a space of constant curvature, 
i.\,e.~$P^3= \R^3, S^3$, or $\mathbb{H}^3$. Their isometry
groups are well-known, and their universal coverings are
\bdm
G\ =\ S^3\ltimes \R^3,  \ \ 
\widetilde{\SO(4)}=S^3\x S^3, \ \ \text{ or }\  \SL(2,\C).  
\edm
The semidirect product  corresponds to the singular case
$4\beta(\alpha-2\beta) - {\alpha'}^2=0$ that we had postponed, while
the other two correspond to non-singular choices of $\alpha,\alpha',$
and $\beta$ as stated before. To see explicitly how the Lie algebra
satisfying $4\beta(\alpha-2\beta)= {\alpha'}^2\neq 0$ is isomophic
to the semidirect product  $S^3\ltimes \R^3$, one performs the
basis change 
\bdm
g_i\ :=\  e_i+ \frac{\alpha'}{2(\alpha-2\beta)}\, \Omega_{i-1},\quad i=2,4,6,
\quad \Omega_1,\ \Omega_3,\ \Omega_5\ \text{ unchanged}
\edm
and checks that $[g_i,g_j]=0$ for $i,j=2,4,6$ and that the elements
of $\Lin( \Omega_1,\Omega_3,\Omega_5)\cong \so(3)$ act on $g_j$ by the 
adjoint representation.
We summarize the result:
\begin{thm}[case D.2 -- classification]\label{thm.n=6-D.2-class}
%
A  complete, simply connected Riemannian $6$-manifold $(M^6, g, T)$
with parallel
skew torsion $T$, $\rk(* \sigma_T)=6$ and $\ker T = 0$
that is \emph{not} isometric to a nearly K\"ahler manifold
is one of the following Lie groups with a suitable family of 
left-invariant metrics:
\begin{enumerate}
\item The nilpotent Lie group with Lie algebra  $\R^3\x \R^3$
with commutator $[(v_1, w_1), (v_2, w_2)]\ =\ (0, v_1\x v_2)$
(see \cite{Sch07}),
\item the direct or the semidirect product of $S^3$ with  $\R^3$,  
\item the product $S^3\x S^3$  (described in Section \ref{S3xS3}),
\item the Lie group $\SL(2,\C)$ viewed as a $6$-dimensional
real manifold  (described in Section \ref{SL(2,C)}).
\end{enumerate}  
\end{thm}  
\begin{NB}
Large families of half-flat almost complex structures on 
$S^3\x S^3$ were constructed in \cite{Schulte-Hengesbach10} and \cite{MaSa};
they overlap those described in Section \ref{S3xS3}, but it is not evident
how to test which have parallel torsion or curvature. 
The nilpotent example (1) is rather exceptional:  as explained in
\cite{Sch07}, nilmanifolds have almost never parallel torsion.
\end{NB}
%
\section{Explicit realisation of the occurring naturally reductive spaces}
%
This section compiles all naturally reductive homogeneous spaces of
dimension five and six (with the exception of a few degenerate, 
not-so-interesting cases like most products etc.); 
some of them occur in families that generalise to higher dimensions.
%
\subsection{The Stiefel $5$-manifold}\label{Stiefel-type}
%
In what follows, we will study the homogeneous space 
\bdm
M^5 \ =\  \dfrac{\SO(3) \times \SO(3)}{\SO(2)_r}.
\edm
It appears in the 5-dimensional classification as the compact
representative of case (2) in the Classification  
Theorem \ref{thm.B1-classification}.  Here, we denote by $\SO(2)_r$ the subgroup of 
$\SO(3)\times \SO(3)$ consisting of products of matrices of the form 
\bdm
g(t)\, = \, \left( \begin{bmatrix} \cos t & - \sin t & 0\\ 
\sin t & \cos t & 0 \\ 0 & 0 & 1  \end{bmatrix}, 
\begin{bmatrix} \cos rt & - \sin rt & 0\\ \sin rt & \cos rt & 0 \\ 
0 & 0 & 1  \end{bmatrix} \right),
\edm
where $r$ is a rational parameter. The condition $r\in \Q$ guarantees 
that $\SO(2)_r$ is closed inside $\SO(3)\times \SO(3)$. 
Let $\{e_{12}, e_{13}, e_{23}\}$ be the standard basis of $\so(3)$.
The Lie algebra $\so(2)_r=:\h$ embeds inside 
$\so(3)\oplus\so(3)=:\g$ as the one-dimen\-sional Lie algebra generated by
$h:=(e_{12}, re_{12})$. As invariant scalar product, we choose
the multiple of the Killing form $K(X,Y) := \frac{1}{2} \tr(X^t Y)$, so 
that the $e_{ij}$ have norm 1. Denote by $\m$
the linear span of the elements
\bdm
x_1 = (e_{13},0), \ x_2= (e_{23},0), \  
y_1 = (0, e_{13}), \ y_2=(0, e_{23}), \  z = \frac{1}{\sqrt{r^2+1}} 
(-r e_{12}, e_{12}).
\edm
They are an orthornomal basis of $\m$  and one checks
that $\m$ is a reductive orthogonal complement of $\h$ in $\g$, $\g = \h 
\oplus \m$.
Let us now introduce two parameters in our metric, which we shall call 
$g_{a,b}$, $a, b >0$, or simply $g$ if the parameters are understood,  
by prescribing that the following is an orthonormal basis
$$ u_1 = \sqrt{a}\, x_1, \quad u_2 = \sqrt{a}\, x_2,\quad v_1 = \sqrt{b}\, 
y_1, \quad v_2 = \sqrt{b}\, y_2, \quad \xi = z  $$
for a fixed pair of parameters $a,b>0$. Of course, we could introduce a third
parameter $c>0$ of the metric in direction $z$, but we normalize it to $1$.
As for the projection of the commutators of elements of $\m$ to both 
$\h$ and $\m$ we have the following non-vanishing terms
\bdm
\begin{array}{lll} 
\,[u_1, u_2] = - \frac{a\, r}{\sqrt{r^2+1}}\,\xi +\frac{a}{r^2+1}\, h &
[u_1,\xi] = \frac{r}{\sqrt{r^2+1}}\, u_2, & 
[u_2, \xi] =  -  \frac{r}{\sqrt{r^2+1}}\, u_1,\\ 
\, [v_1, v_2] =  \frac{b}{\sqrt{r^2+1}}\,  \xi +\frac{b\,r}{r^2+1}\, h & 
[v_1, \xi] = -\frac{1}{\sqrt{r^2+1}}\, v_2,  & 
[v_2, \xi] =  \frac{1}{\sqrt{r^2+1}}\, v_1.  \\ 
\end{array}
\edm
Comparing the example  with the Classification Theorem 
\ref{thm.B1-classification},
we see that the parameters are related by the formulas 
$\lambda = \frac{r a}{\sqrt{r^2+1}}$ 
and $\varrho = -\frac{b}{\sqrt{r^2+1}}$. 
Observe that, as it stands, we can only conclude that we have a 
naturally reductive metric for $a=b=1$ (see Remark \ref{NBIsotropy}).
To describe the
isotropy representation $\SO(2)_r \ra \SO(\m)$, we identify 
$\m$ with $\R^5$, and choose the ordering  $\{u_1,u_2,v_1,v_2,\xi \}$. 
Denote by 
$E_{i,j}$, ($i< j,\ i,j =1, \ldots, 5$) the standard basis of 
$\so(5)$, see Remark \ref{notations}.
The linear isotropy representation $\lambda: \h \lra \mathfrak{so}(\m)$
may then be expressed as
\bdm
\lambda (h) = E_{12} + r\, E_{34}.  
\edm
Let $\alpha_i$, $\beta_i$  and $\eta$ be the dual forms of $u_i$, $v_i$ 
and $\xi$, respectively ($i=1,2$). The vector $\xi$ is fixed under the 
isotropy representation, so it defines a global vector field which 
gives a  preferred direction in the tangent bundle. The forms 
$\eta$, $\alpha_1 \wedge \alpha_2$ and $\beta_1 \wedge \beta_2$ are also 
globally defined. Hence the tensor 
\bdm
\varphi = \alpha_1 \otimes u_2 - \alpha_2 \otimes u_1 
+ \beta_1 \otimes v_2 - \beta_2 \otimes v_1
\edm
is a well-defined almost contact structure which can be readily checked 
to be compatible with the metric $g_{a,b}$. The fundamental form $F$ is 
given by $F= - (\alpha_ 1 \wedge \alpha_2 + \beta_1 \wedge \beta_2)$
 which is always closed. The exterior derivative of the contact form 
$\eta$  is given by
\bdm
d\eta = \frac{a r}{\sqrt{r^2+1}}\, \alpha_1 \wedge \alpha_2 
- \frac{b}{\sqrt{r^2+1}}\, \beta_1 \wedge \beta_2.
\edm
Hence, $F$ and $d\eta$ are proportional if and only if $a r +b=0$, whereas 
the Nijenhuis tensor vanishes for all values. 
Therefore, $M^5$ is a quasi-Sasaki manifold, and it is
$\alpha$-Sasaki if $ar+b=0$. 
We can calculate the map $\Lambda^g: \ \m\ra\so(\m)$ describing the 
Levi-Civita connection  by means of the expression \cite[Ch.X]{Kobayashi&N2} 
\bdm
\Lambda^{g}(X) Y = \frac{1}{2} [X,Y]_\m + \frac{1}{2} U(X,Y),
\edm
where $U$ is such that $g(U(X,Y),Z) = g([Z,X],Y) + g([Z,Y],X)$.  
Then we have the following connection forms
\bdm
\begin{array}{lll} \Lambda^g(u_1) = -\frac{ar}{2 \sqrt{r^2+1}} E_{25}, &
  \Lambda^g(u_2) = \frac{ar}{2 \sqrt{r^2+1}} E_{15}, & 
\Lambda^g(v_1) = \frac{b}{2 \sqrt{r^2+1}} E_{45},\\ 
\Lambda^g(v_2) = - \frac{b}{2 \sqrt{r^2+1}} E_{35}, & \Lambda^g(\xi) =
\frac{r(a-2)}{2\sqrt{r^2+1}} E_{12} + \frac{2-b}{2 \sqrt{r^2+1}} E_{34}, &
\end{array}
\edm
and we see that $\xi$ is a Killing field. Consequently, 
the quasi-Sasaki structure
admits a characteristic connection $\nabla$ whose torsion is given by 
\cite{FI02} 
\bdm
T= \eta \wedge d\eta =  \frac{a r}{\sqrt{r^2+1}}\, \eta \wedge \alpha_1 
\wedge \alpha_2 - \frac{b}{\sqrt{r^2+1}}\,\eta \wedge \beta_1 
\wedge \beta_2,
\edm
and it is described by the linear map $\Lambda$ 
\bdm
\Lambda(u_1) = \Lambda(u_2) = \Lambda(v_1) = \Lambda(v_2) = 0,\quad 
\Lambda(\xi) = \frac{r (a-1)}{\sqrt{r^2+1}} E_{12} 
+ \frac{1-b}{\sqrt{r^2+1}} E_{34}.
\edm
We immediately see (as expected) that $T$ is indeed parallel. For the 
curvature tensor $\kr : \Lambda^2 \rightarrow \Lambda^2$ we use the 
following formula
$$\kr(X,Y) = [\Lambda(X), \Lambda(Y)] - \Lambda([X,Y]_\m) -
 \lambda([X,Y]_\h) \ .$$ 
Its image is contained in the $2$-dimensional abelian Lie algebra 
generated by $\alpha_1 \wedge \alpha_2$
and $\beta_1 \wedge \beta_2$ 
$$\kr = \frac{a (a-1) r^2 - a}{r^2+1} (\alpha_1 \wedge \alpha_2)^2 
+ \frac{b(b-1) - b r^2}{r^2+1} 
(\beta_1 \wedge \beta_2)^2 
- \frac{abr}{r^2+1} (\alpha_1 \wedge \alpha_2 \odot \beta_1 \wedge \beta_2).$$
We have $\nabla \kr = 0 = \nabla T$, i.e. a 2-parameter 
family of naturally reductive metrics on the manifold $M^5$ (a third 
parameter $c$ is allowed,
but yields only a global rescaling of the metric). 
We can easily compute the Ricci tensor for $\nabla^g$, it yields 
$$\Ric^g = \mathrm{diag}\left(a - \frac{a^2 r^2}{2(r^2+1)}, a 
- \frac{a^2 r^2}{2(r^2+1)}, b - \frac{b^2 }{2(r^2+1)}, 
b - \frac{b^2 }{2(r^2+1)}, \frac{a^2 r^2}{2(r^2+1)} 
+ \frac{b^2}{2(r^2+1)} \right).$$ 
If we look for Einstein metrics, we get a system of two equations. A detailed
discussion tells that $a$ can be expressed through $b$, and
that $b$ is the unique real solution of a cubic equation
\bdm
a\ =\ b \left[ 2 - \frac{3}{2(r^2+1)} b\right],\quad
r^2 b \left[ 2 - \frac{3}{2(r^2+1)} b\right]^2 \ =\ 
2 (r^2+1) - 2b.
\edm
Hence there is exactly one pair of values $(a,b)$ for any parameter
value of $r^2$. For example, we check that  $r^2 =1$ yields
 $a=b =4/3$, and this is the 
Einstein-Sasaki metric of \cite{Jensen75} and 
\cite{Friedrich80}.
%
\subsection{The Berger sphere $S^5=\U(3)/\U(2)$}\label{berger}
%
The Berger sphere (and its non-compact sibling $\U(2,1)/\U(2)$) is the only 
$5$-dimensional naturally reductive space that is $\alpha$-Sasakian (case B.2),
but not part of a quasi-Sasakian family (case B.1), see 
Theorem \ref{n=5EWgleich}. We sketch two alternative ways to view the Berger 
sphere; since this is a relatively well-known example, we shall be brief.
Consider the pair of Lie algebras
\bdm
\mathfrak{u}(3) \ \cong \ \Big\{A \in \mathcal{M}_3(\C) \, : \, A  
+  \bar{A}^t  =  0 \Big\}, \quad
\mathfrak{u}(2) \ \cong \ \left\{  \begin{bmatrix} 0 & 0 \\ 
0 & B   \end{bmatrix} \, : \ B \in \mathcal{M}_2(\C)  , \,  
B  +  \bar{B}^t \ = \ 0 \right\}.
\edm
First, we show that the realisation of $S^5$ as the quotient $\U(3)/\U(2)$
is already naturally reductive (in the traditional sense). Indeed,
there are two elements commuting with $\mathfrak{u}(2)$ and a complementary 
$4$-dimensional subspace,
\bdm
Z_1 \ = \  \begin{bmatrix} i & 0 & 0 \\ 
0 & i & 0 \\ 0 & 0 & i  \end{bmatrix} , \quad
Z_2 \ = \  \begin{bmatrix} 0 & 0 & 0 \\ 
0 & i & 0 \\ 0 & 0 & i  \end{bmatrix} , \quad
\mathfrak{m}_0 \ = \ \left\{ \begin{bmatrix} 0 & - \bar{v}^t \\ 
v & 0   \end{bmatrix} \, : \ v \in \C^2 \right\} .
\edm
Any subspace of the family
\bdm
\mathfrak{m}_{\alpha, \beta} \ = \ \mathfrak{m}_{0}  \oplus \R \cdot 
(\alpha Z_1  + \beta Z_2) , \quad \alpha \ \not= \ 0 
\edm
decomposes the Lie algebra $\un(3) = \un(2) \oplus \m_{\alpha, \beta}$,
and $[\un(2) , \m_{\alpha, \beta}] \subset \m_{\alpha, \beta}$ holds. The scalar 
product on the complement $\mathfrak{m}_{\alpha, \beta}$ is defined by a 
bi-invariant scalar product on $\mathfrak{u}(3)$.
This is a $2$-parameter family of naturally reductive structures, see Theorem 
\ref{n=5EWgleich}. Alternatively, one can realize the $5$-sphere
as $S^5=\SU(3)/\SU(2)$ with the same embeddings as above, now with vanishing
trace. Then one chooses the reductive decomposition
\bdm
\su(3)\ =\ \su(2)\oplus\m, \quad  \m = \m_0\oplus \langle \eta\rangle\ 
\text{ with } \eta= \frac{1}{\sqrt{3}}\, \diag (-2i,i,i)\ =\ 
\frac{1}{\sqrt{3}}\, (-2 Z_1+3Z_2).
\edm
As basis of $\m_0$, we choose the elements $e_1,\ldots,e_4$ corresponding
to the vectors $v=(1,0)$, $(i,0)$, $(0,1)$, $(0,i)\in \C^2$. With
respect to the Killing form $\beta(X,Y)=-\tr(XY)/2$ of $\su(3)$,
the vectors $e_1, \ldots, e_3, \eta$ are orthonormal. Thus,
we can define a deformation of the scalar product by
\bdm
g_\gamma \ := \ \beta\big|_{\m_0} \oplus \frac{1}{\gamma}
\beta\big|_{\langle \eta\rangle} , \quad \gamma>0.
\edm
Again, a second parameter could be introduced by allowing a rescaling
on $\m_0$, but this has no intrinsic geometrical meaning. Now one checks
that $F := e_{12}+e_{34}$ with Reeb vector field 
$\tilde\eta=\eta/\sqrt{\gamma}=:e_5$
defines an $\alpha$-Sasaki structure with characteristic connection
$\nabla$; as an invariant connection, $\nabla$ is described by
the map $\Lambda: \ \m\ra\so(\m)$  (see \cite[Ch.X]{Kobayashi&N2}) 
given by $\Lambda(e_i)= 0$ for $i=1,\ldots,4$ and 
$\Lambda(e_5)= (\sqrt{3/\gamma}-\sqrt{3\gamma})(E_{12}+E_{34})$.
The torsion and  curvature of the connection $\nabla$ are given by
\bea[*]
T & =& \tilde\eta\wedge d\tilde \eta\ =\ 
\sqrt{3/\gamma} (e_{12}++e_{34})\wedge e_5, \\
\kr & = & \left[\frac{3}{\gamma}-3\right](e_{12}+e_{34})^2 -
[(e_{13}+e_{24})^2 + (e_{14}-e_{23})^2 + (e_{12}-e_{34})^2].
\eea[*]
Comparing with Theorem \ref{n=5EWgleich} and formula ($\ref{kr.n=5EWgleich}$),
we see that $a=-1, b=3 / \gamma-3$ and the coefficient $\vrho=\lambda$
of the torsion  is given by $\vrho^2=3/\gamma$.
%
%
One verifies that $T$ and  $\kr$ are indeed
$\nabla$-parallel. For $\gamma=1/2$, the $\nabla$-Ricci curvature
vanishes, while for $\gamma=3/4$, the metric is Riemannian Einstein.
This is a particular case of the deformation of Einstein-Sasaki
metrics that yield $\nabla$-Ricci-flat connections described in
\cite{Agricola&F12}.
%
\subsection{The Heisenberg group of dimension $2n+1$}\label{odd-Heisenberg}
%
The Heisenberg group of dimension $2n+1$ is the subgroup of $GL(n+2,\R)$ 
given by upper triangular matrices of the form
$$H^{2n+1} = \left\{  \begin{bmatrix} 1 & x^t & z \\ 
0 & 1 & y \\ 0 & 0 &1 \end{bmatrix}\ :\ x,y\in\R^n,\, z\in \R\right\}.$$ 
It appears in the classifications in dimension $3$ (Theorem \ref{thm.class-dim3})
and dimension $5$ (Theorem \ref{thm.B1-classification}, case B.1).
Clearly $H^{2n+1}$ is diffeomorphic to $\R^{2n+1}$. It can be readily checked 
that the following sets are, respectively, a basis of left-invariant vector 
fields and its dual basis of left-invariant $1$-forms,
\bdm
\left\{ \frac{\partial}{\partial x_1}, \frac{\partial}{\partial y_1}+x_1 
\frac{\partial}{\partial z},\frac{\partial}{\partial x_2}, 
\frac{\partial}{\partial y_2}+x_2 \frac{\partial}{\partial z} \cdots, 
\frac{\partial}{\partial z}\right\},\quad 
\left\{ dx_1, dy_1, dx_2, dy_2, \cdots, dz - \sum_{i=1}^n x_i dy_i \right\}.
\edm
We have an odd-dimensional manifold with a clearly preferred direction and we 
can define a contact structure as follows. Choose 
$\xi = \frac{\partial}{\partial z}$ to be the Reeb field and 
$\eta = dz- \sum_{i=1}^n x_i dy_i $ to be the contact form. The (1,1)-tensor 
$$\varphi = \sum_{i=1}^n \left( dx_i \otimes \left( 
\frac{\partial}{\partial y_i} + x_i \frac{\partial}{\partial z} \right) 
- dy_i \otimes \frac{\partial}{\partial x_i} \right)$$ 
is then a contact structure on $H^{2n+1}$, since  
$\varphi^2 = - \mathrm{Id} + \eta \otimes \xi$. For each $i=1,\cdots, n$ 
let $\lambda_i$ be a positive scalar and consider the $n$-tuple 
$\lambda=(\lambda_1, \cdots, \lambda_n)$. Then
$$g_\lambda = \sum_{i=1}^n \frac{1}{\lambda_i} (dx_i^2 + dy_i^2) 
+ \left[ dz- \sum_{j=1}^n x_j dy_j\right]^2$$ 
defines an $n$-parameter family of metrics which are compatible with 
$\varphi$. Hence 
$(H^{2n+1}, g, \varphi, \xi, \eta)$ is an almost contact metric manifold.
An orthonormal frame $\{u_i, v_i, \xi\}$ of $TH^{2n+1}$ 
and its  dual frame $\{\alpha_i, \beta_i, \eta\}$ are given by  
\bdm
u_i = \sqrt{\lambda_i} \frac{\partial}{\partial x_i},\quad
v_i = \sqrt{\lambda_i} \left(\frac{\partial}{\partial y_i} 
+ x_i \frac{\partial}{\partial z}\right) \text{ and }
\alpha_i = \frac{1}{\sqrt{\lambda_i}} dx_i, \quad  
\beta_i = \frac{1}{\sqrt{\lambda_i}} dy_i \text{ for } i=1,\cdots,n.
\edm
The only  non-vanishing commutators are  $[u_i, v_i] = \lambda_i \xi$, 
for all $i=1,\cdots, n$.
An easy but tedious computation shows that the Nijenhuis tensor $N$ of 
$\varphi$ vanishes. Moreover, $d\eta$ and the fundamental form 
$F(X,Y) = g(X,\varphi(Y))$ are expressed by
$$d\eta = - \sum_{i=1}^n \lambda_i \alpha_i \wedge \beta_i \quad \mbox{ and } \quad F = - \sum_{i=1}^n \alpha_i \wedge \beta_i.$$
Therefore $F$ is always closed and $H^{2n+1}$ is quasi-Sasaki for all 
parameters. Furthermore $H^{2n+1}$ is $\alpha$-Sasaki if and only if all 
parameters coincide and, in particular, Sasaki for $\lambda_1 , \, \ldots \,  
\lambda_n = 2$.
Using the first structure equation of Cartan, we can see that the 
non-vanishing Levi-Civita connection forms are given by
$$\omega^g_{u_i v_i} = -\frac{\lambda_i}{2} \eta,\quad \omega^g_{u_i\xi} 
= -\frac{\lambda_i}{2}\beta_i, \quad \omega^g_{v_i\xi} 
=\frac{\lambda_i}{2}\alpha_i.$$ 
Thus our Reeb field $\xi$ is a Killing vector field and since $N=0$, a 
theorem from \cite{FI02} guarantees the 
existence of the characteristic connection whith torsion $3$-form
$T= \eta \wedge d\eta$.
The  non-vanishing connection forms for $\nabla$ are then simply
$\omega_{u_iv_i} = - \lambda_i \eta\ ( i=1,\ldots,n)$. 
Using now the second structure equation of Cartan, we can calculate the 
curvature forms for the characteristic connection, 
$\Omega_{u_iv_i} = - \lambda_i \, d\eta$. Equivalently, the curvature tensor 
can be written as
%
$$
\kr = \sum_{i \leq j}^n \lambda_i \lambda_j (\alpha_i\wedge \beta_i) 
\odot (\alpha_j \wedge \beta_j).$$ 
Observe that the holonomy algebra of $\nabla$ is one-dimensional,
and therefore in particular abelian. We can readily 
check that $\nabla$ satisfies
$\nabla T = 0 = \nabla \kr$ and therefore we have determined an $n$-parameter 
family of naturally reductive homogeneous structures on the Heisenberg 
group of dimension $2n+1$. 
\begin{NB}
Heisenberg groups of dimension $2n+1\geq 5$ with their
naturally reductive structure 
defined above are the first known examples of manifolds with parallel 
skew torsion carrying a Killing spinor with torsion that do \emph{not}
admit a Riemannian Killing spinor (in fact, they do not even carry a
Riemannian Einstein metric, which would be a necessary requirement
for such a spinor) -- see \cite{ABK12} and \cite{Agricola&H14}.
This is a typical example of how naturally reductive spaces 
are used in differential geometry as a vast reservoir of examples.
\end{NB}
%
\subsection{The Lie group $S^3\x S^3$}\label{S3xS3}
%
In this section, we explain the naturally reductive structures on
$S^3 \x S^3$ with $\hol^\nabla=\so(3)$ (case D.2, Theorem \ref{thm.n=6-D.2-class}).
We realise $S^3 \x S^3$ as the homogeneous space $G/H$ where 
$G= \SU(2)\x \SU(2) \x SU(2)$ and $H= \SU(2)$ is embedded into $G$ 
diagonally, that is,
\bdm
H\ =\ \{(h, h, h): h \in \SU(2) \}.
\edm
Let $\g = \su(2)\oplus \su(2) \oplus \su(2)$ be the Lie algebra of $G$ and 
$\h = \su(2) = \{ (C, C, C): C \in \su(2) \}$ be the Lie algebra of $H$. 
Consider the following spaces
\bdm
\m_1 \,= \, \{ (A, aA, bA): a, b \in \R, A \in \su(2) \},\quad
\m_2 \,= \, \{ (B, cB, dB): c, d \in \R, B \in \su(2) \}
\edm
and let $\m$ be the direct sum of $\m_1$ and $\m_2$. Then $\m$ is a 
reductive complement of $\h$ inside $\g$ if and only if 
\bdm
\Delta := \det \left[ \begin{array}{ccc} 1 & 1 & 1 
\\ 1 & a & b \\ 1 & c & d \end{array}
 \right] = (a-1) (d-1) - (b-1) (c-1)\ \neq 0.
\edm
Let $K(X,Y) = -\frac{1}{2} \tr(XY)$ denote the (rescaled) Killing 
form on $\su(2)$  and define an inner product on 
$\m$, for each parameter $\lambda >0$, as
$$\langle (A_1+B_1, a A_1 + c B_1, b A_1 + d B_1), 
(A_2+B_2, a A_2 + c B_2, b A_2 + d B_2) \rangle
= K(A_1, A_2)+ \frac{1}{\lambda^2} K(B_1,B_2).$$ 
We define an almost complex structure $J$ on $\m$ by 
\bdm
J((A, a A, b A)+(B, c B, d B) ) = -\frac{1}{\lambda} (B, a B, b B) 
+ \lambda (A, c A, d A).
\edm
Let 
\bdm
Y_1\ =\ \begin{bmatrix}i & 0 \\ 0 & -i\end{bmatrix},\
Y_3\ =\ \begin{bmatrix}0 & -1 \\ 1 & 0\end{bmatrix},\
Y_5\ =\ \begin{bmatrix}0 & i \\ i & 0\end{bmatrix},\
\edm
be the standard basis of $\su(2)$. Recall that we 
have the following commutator relations
\bdm
[Y_1, Y_3] = - 2 Y_5, \quad [Y_1, Y_5] = 2 Y_3, \quad [Y_3, Y_5] = - 2 Y_1.
\edm
Take $h_i = (Y_i, Y_i, Y_i)$, $i=1,3,5$. Then $\{ h_1, h_3,h_5 \}$ is a basis 
of $\h$. Consider also the following orthonormal basis of $\m$
\bdm
e_i = (Y_i, a Y_i, b Y_i)\ \   \text{ and } \  
e_{i+1} = \lambda (Y_i, c Y_i,d Y_i)\ \ \text{ for }i=1,3,5.
\edm
Remark that $J$ is given, in this basis, by the $2$-form 
$\Omega = -(e_{12}+ e_{34} + e_{56})$. The isotropy representation is given by 
\bdm
\ad h_1\, =\, -2(E_{35}+E_{46})\, =:\, -2A_1 ,\ \ 
\ad h_3\, =\, 2(E_{15}+E_{26})\, =:\, 2A_3,\ \ 
\ad h_5\, =\, -2(E_{13}+E_{24})\, =:\, -2A_5.
\edm
The commutator structure is somewhat complicated. For ease of notation 
we introduce the following coefficients
\bdm
\begin{array}{ll}
\mu = -\frac{2}{\Delta} ((a^2-1)(d-1) - (b^2-1)(c-1)) \qquad 
& \nu = -\frac{2}{\Delta} (b-1)(a-1)(b-a) \\ 
\gamma = -\frac{2}{\Delta} ( a(d-b^2)+ a^2(b-d)+ (b^2-b) c) & \delta = -\frac{2}{\Delta} ( c (a (d-1) - bd +1) + (b-1) d) \\ 
\sigma = \frac{2}{\Delta} ( (a-1)(1-bd) + (ac-1)(b-1) )& \tau = \frac{2}{\Delta} ( ac (d-b) + cb (1-d) + ad (b-1) )\\
\xi = - \frac{2}{\Delta} (c-1)(d-1)(c-d) & \eta = -\frac{2}{\Delta} ((d^2-1) (a-1) - (c^2-1)(b-1))\\
\theta = -\frac{2}{\Delta} ( d^2(c-a) + c^2(b-d) + (da-cb)).
\end{array}
\edm
Then we can write the nonvanishing brackets as

\begin{center}
\begin{tabular}{ll}
$[ e_1, e_3] = \mu e_5 + \frac{\nu}{\lambda} e_6 + \gamma h_5, 
\quad $ 
& $[e_1, e_4] = [e_2, e_3] = \lambda \delta e_5 + \sigma e_6 + \lambda \tau h_5, \quad$ \\  
$[e_1, e_5] = -\mu e_3 - \frac{\nu}{\lambda} e_4 + \gamma h_3, 
\quad$ 
&
$[ e_1, e_6] = [e_2, e_5] = -\lambda \delta e_3 - \sigma e_4 - \lambda \tau h_3,  $ \\
$[e_2, e_4] = \lambda^2 \xi e_5 + \lambda \eta e_6 + \lambda^2 \theta h_5, 
\quad$ & $[e_2, e_6] = -\lambda^2 \xi e_3 - \lambda \eta e_4 
- \lambda^2 \theta h_3, $\\
$[ e_3, e_5] = \mu e_1 + \frac{\nu}{\lambda} e_2 + \gamma h_1, 
\quad $ &  $[e_3, e_6] = [e_4, e_5] \lambda \delta e_1 + \sigma e_2 
+ \lambda \tau h_1, \quad$\\
$[e_4, e_6] = \lambda^2 \xi e_1 + \lambda \eta e_2 
+ \lambda^2 \theta h_1. \quad$
\end{tabular}
\end{center}

The Nijenhuis tensor $N$ is totally skew-symmetric and given by 
\bdm
N= [\lambda^2 \xi + 2 \sigma - \mu] (e_{135}-e_{146}-e_{236}-e_{245}) 
+ \left[ \frac{\nu}{\lambda} + \lambda (2\delta - \eta) \right] 
(e_{246}-e_{136}-e_{145}-e_{235}).
\edm
We can also compute that
\bdm
d^c\Omega = -3 \lambda^2 \xi e_{135} - 3 \frac{\nu}{\lambda} e_{246} 
+ (2\sigma - \mu) (e_{146}+e_{245}+e_{236})+ \lambda (2 \delta - \eta) 
(e_{145}+e_{136}+e_{235}).
\edm
Therefore the torsion tensor $T = N+ d\Omega\circ J$ of the almost complex
structure is \cite[Thm 10.1]{FI02}
\bdm
T  =  [-2 \lambda^2 \xi + 2 \sigma - \mu] e_{135} 
+ \left[-2 \frac{\nu}{\lambda}+ \lambda (2 \delta - \eta)\right] 
e_{246}  - \lambda^2 \xi (e_{146}+ e_{236}+ e_{245}) 
- \frac{\nu}{\lambda} (e_{136}+ e_{145}+ e_{235}).
\edm
For all parameters, the Hermitian structure 
is of type $\mathcal{W}_1\oplus \mathcal{W}_3$.
Its characteristic connection $\nabla$ is given by the map 
$\Lambda: \m \lra \so(\m)$
\bdm
\Lambda(e_i)\, =\, (-\lambda^2 \xi + \sigma ) A_i \ 
\text{ and } \  \Lambda (e_{i+1})\, = \, 
\left( -\frac{\nu}{\lambda} + \lambda \delta \right) A_i
\text{ for } i=1,3,5.
\edm
It is then a straightforward computation to check that $\nabla T = 0$. 
As for the curvature tensor, we obtain 
\bdm
\kr\ = \  \Sigma\, [ A_1^2+ A_3^2+A_3^2],
\quad \text{where} \ \ \Sigma 
:= \frac{\nu^2}{\lambda^2} + \lambda^4 \xi^2 
- \lambda^2 \xi (2 \sigma - \mu)- \nu (2 \delta - \eta).
\edm
Thus, $\kr$ is a projection on  $\Lin(A_1,A_3,A_5)=\so(3)=\hol(\nabla)$, 
compare with equation (\ref{n=6-curvature}). This shows
(and one easily checks)  that $\nabla \kr = 0$. 
The constant $\Sigma$ also appears in $\sigma_T$,  for we have 
$\sigma_T = - \Sigma (*\Omega)$. Therefore $\kr = 0$ if and only if 
$\sigma_T = 0$.

In this description, some parameters play no significant geometrical role.
The torsion $T$ is  a linear combination of the 
four $\SO(3)$-invariant $3$-forms  in $\mathcal{W}_1\oplus\mathcal{W}_3$,
as given in the proof of Theorem \ref{Dn=6, classification}.
A suitable base change would make the last term  disappear, but
we can obtain the same result by choosing our parameters $a,b,c$, and $d$ 
so that the coefficient $\nu$ vanishes. The choice $a=b=1$ is not 
allowed, since it yields $\Delta=0$, but either $a=1$ or $b=1$ is
an admissible choice. So let us set $b=1$ ($\Delta\neq 0$ is then
equivalent to $a\neq 1$ and $d\neq 1$).  Most, but-alas-not all,
coefficients simplify considerable.

Let us identify a few particularly interesting choices of parameters.
First, there are several solutions for $\Sigma=0$, the
simplest being $c=1$ or $c=d$ (because both imply $\xi=0$).
Recall that if  $\alpha,\alpha',\beta$ denote the
coefficients of the torsion as in equation (\ref{n=6-curvature}),
we described in the proof of Theorem \ref{Dn=6, classification}
and Remark \ref{NB-curvature} three noteworthy  situations.
They all require $\alpha'=0$, which is equivalent to 
$2c=d+1$. The discussion of the additional conditions is
summarized in the following table:

\bdm
\ba{|l|l|l|}\hline
\text{Geometric description} & \text{condition} 
& \text{equivalent to}\\ \hline \hline
 \text{bi-inv. Riemannian Einstein metric}  & \alpha= +\beta 
& \lambda= \frac{2 |a-1|}{|d-1|}\\ \hline
 \text{pure type } 
\mathcal{W}_1 \text{(nearly K\"ahler)} & \alpha=-\beta 
& \lambda= \frac{2 |a-1|}{\sqrt{3}|d-1|}\\ \hline
 \text{pure type } 
\mathcal{W}_3  & \alpha=3\beta 
& \text{impossible}\\ \hline
\ea
\edm

\medskip
The next example has exactly the opposite behaviour:
the naturally reductive metrics on
$\SL(2,\C)$ are either of type $\mathcal{W}_3$ or 
$\mathcal{W}_1\oplus \mathcal{W}_3$, but never of pure type
$\mathcal{W}_1$.
%
\subsection{The Lie group $\SL(2,\C)$}\label{SL(2,C)}
%
The complex Lie group $\SL(2,\C)$ can be understood as a
real $6$-dimensional non compact manifold. Its standard
complex structure and Killing form can be deformed to yield an
almost Hermitian structure with parallel torsion of Gray-Hervella class
$\mathcal{W}_1\oplus \mathcal{W}_3$. Its $\mathcal{W}_3$ structure
was first discovered in \cite{Alexandrov&F&S05}, and
enlarged to $\mathcal{W}_1\oplus \mathcal{W}_3$ (albeit rather laconically) 
in \cite{Sch07}. Since this is a rather unusual
(and quite tricky) case of a naturally reductive space, we will
give a self-contained account of its geometry here. Recall that
\bdm
\sl(2,\C) \ = \ \Big\{A \in \mathcal{M}_2(\C) \, : \, \tr A =  0 \Big\}
\ =\ \su(2) \oplus i\,\su(2).
\edm
We realize $\SL(2,\C)$ as the quotient $G/H=\SL(2,\C)\x \SU(2)/\SU(2)$ with
$H=\SU(2)$ embedded diagonally and a reductive complement $\m_\alpha$ of $\h$
inside $\g=\sl(2,\C)\oplus \su(2)$ depending on a real parameter 
$\alpha\neq 1$, i.\,e.
\bdm
\h \ =\ \{ (B,B)\, :\, B\in\su(2)\},\quad
\m_\alpha \ :=\ \{(A+\alpha B,B)\, :\, A\in i\, \su(2),\ B\in \su(2)\}.
\edm
Observe that elements of the form $(A+\alpha_1 B, \alpha_2 B)$ would still
define a reductive complement of $\h$ (as long as $\alpha_1\neq \alpha_2$),
but $\alpha_2=0$ is uninteresting, since $\m$ would then be a subalgebra
of $\g$. Hence, our Ansatz for $\m_\alpha$ sets this second
constant equal to one. 
We choose as basis of  $\sl(2,\C)$ over the reals
$Y_1,\ Y_3,\ Y_5,\ 
Y_2 = i\, Y_1,\ Y_4= i\, Y_3,\ Y_6 = i\, Y_5$, where
the first three elements are defined as in the previous Example
\ref{S3xS3}.
Thus, the elements $h_i=(Y_i,Y_i), \ i=1,3,5$ form a basis of $\h$. The
Killing form $2 \, \beta(X,Y)=\Re\tr (XY)$ of $\sl(2,\C)$ is negative
definite on $\su(2)$, positive definite on $i\,\su(2)$ and these two
spaces are orthogonal. Therefore, the formula
\bdm
g_\lambda \big((A_1+\alpha B_1,B_1), (A_2 +\alpha B_2 ,B_2 ) \big)\ :=\
\beta(A_1, A_2) - \frac{1}{\lambda^2}\beta(B_1,B_2),\quad \lambda>0 
\edm
defines a one-parameter family of Riemannian metric on $\SL(2,\C)\cong G/H$.
We shall prove later that $G/H$ with such a  metric is in fact a naturally
reductive space for all $\lambda$. The elements 
\bdm
x_i\,=\, \lambda(\alpha\, Y_i,Y_i),\ i=1,3,5\  \text{ and }\
x_j\, =\, (Y_j,0),\ j=2,4,6  
\edm
form an orthonormal frame of $\m_\alpha$ with respect to
$g_\lambda$. With respect to this frame, the differential 
$\ad:\ \h\ra\so(\m_\alpha)$ of the isotropy representation
is (as in Example \ref{S3xS3}) given by
\bdm
\ad h_1\, =\, -2(E_{35}+E_{46})\, =:\, H_1 ,\ \ 
\ad h_3\, =\, 2(E_{15}+E_{26})\, =:\, H_3,\ \ 
\ad h_5\, =\, -2(E_{13}+E_{24})\, =:\, H_5.
\edm
The non-vanishing commutators of elements in $\m_\alpha$ are
\begin{alignat*}{2}
 [x_1, x_3] & = + 2\alpha\lambda^2 h_5-2\lambda(1+\alpha)x_5, & \qquad
[x_2,x_4]& =  + \frac{2}{1-\alpha} h_5-\frac{2}{\lambda(1-\alpha)}x_5\\
 [x_1, x_5] & = -2\alpha\lambda^2 h_3+2\lambda(1+\alpha)x_3, & \qquad
[x_2,x_6]& = - \frac{2}{1-\alpha} h_3 +\frac{2}{\lambda(1-\alpha)}x_3\\
 [x_3, x_5] & = + 2\alpha\lambda^2 h_1-2\lambda(1+\alpha)x_1, & \qquad
[x_4,x_6]& =  + \frac{2}{1-\alpha} h_1-\frac{2}{\lambda(1-\alpha)}x_1
\end{alignat*}
as well as
\bdm
[x_1,x_4]= [x_2,x_3]= -2\lambda\alpha x_6, \ \ 
[x_1,x_6]= [x_2,x_5]= 2\lambda\alpha x_4, \ \ 
[x_3,x_6]= [x_4,x_5]=-2\lambda\alpha x_2.
\edm
An almost Hermitian structure may be defined by the
K\"ahler form $\Omega :=x_{12}+ x_{34}+x_{56}$; its Nijenhuis tensor turns
out to be a $3$-form,
\be\label{SL2C.N}
N\ = \ 2\left[\lambda(1-\alpha) -\frac{1}{\lambda(1-\alpha)} \right]
[x_{135}-x_{146}-x_{236}-x_{245}].
\ee
By \cite[Thm 10.1]{FI02}, this almost complex structure
admits therefore a unique characteristic connection $\nabla$ with torsion
\bdm
T\ =\ N+d\Omega\circ J\ = \
 \left[2\lambda(1-\alpha) +\frac{4}{\lambda(1-\alpha)} \right]x_{135}
+\frac{2}{\lambda(1-\alpha)}[x_{146}+x_{236}+x_{245}].
\edm
This torsion is $\ad(\h)$-invariant for all parameter values, hence $\nabla T = 0$. 
One checks that this almost complex structure has no contribution
of Gray-Hervella type $\mathcal{W}_4$, and from formula 
($\ref{SL2C.N}$), one can
conclude that it is of type $\mathcal{W}_3$ if and only if 
$\lambda(1-\alpha)= \pm 1$.
The map $\Lambda: \ \m_\alpha\ra\so(\m_\alpha)$ 
characterizing $\nabla$ (see \cite[Ch.X]{Kobayashi&N2}) is given
by
\bdm
\Lambda(x_i )\, =\, \left[ \lambda\alpha -\frac{1}{\lambda(1-\alpha)}\right]
\,H_i\ \text{ for } i=1,3,5 \ \text{ and }\ 
\Lambda(x_j)\, =\, 0 \text{ for } j=2,4,6.
\edm
Let us give the expression for the curvature of $\nabla$:
\bdm
\kr\ =\ 4 \Big( 1 \, + \, \frac{1}{\lambda^2 (1-\alpha)^2}\Big) [ (x_{13}+x_{24})^2+(x_{15}+x_{26})^2+(x_{35}+x_{46})^2],
\edm
which, as should be expected, is nothing else than the projection
on the subalgebra $\so(3)\cong\su(2)$ generated by $H_1, H_3, H_5$, i.\,e.~the
holonomy algebra of $\nabla$.
Thus, the metric $g_\lambda$
is naturally reductive for all parameters.

\appendix\section{The Nomizu construction}\label{app.nomizu}
%
We give here an algebraic construction of an infinitesimal model
of a naturally reductive structure out of a given algebraic curvature 
and a skew torsion. This construction is known by the name
\emph{Nomizu construction} and is used in different places of the literature
(see for example \cite{Tricerri93} and \cite{Cleyton&S04}), but not in the
form that we need for our purpose.

\medskip
Let $\h$ be a real Lie algebra, $V$ a real finite-dimensional $\h$-module
with an $\h$-invariant positive definite scalar product $\langle , \rangle$,
i.\,e.~we assume that $\h\subset \so(V)\cong \Lambda^2 V$.
Let $\kr:\Lambda^2 V \ra\h$ be an $\h$-equivariant map and $T\in(\Lambda^3 V)^\h$
an $\h$-invariant $3$-form. The $4$-form $\sigma_T$ is defined as usual.
We would like to define a Lie algebra structure on $\g:=\h\oplus V$
by setting
\be\label{Nomizu}
[A+X,B+Y]\ :=\ ([A,B]_\h - \kr(X,Y))+ (AY-BX-T(X,Y)),\qquad A,B\in \h,\ X,Y\in V.
\ee
This amounts to finding necessary and sufficient conditions for the
Jacobi identity to hold in $\g$. The first result is classical, hence we omit 
the easy proof:
\begin{lem}\label{lem.nomizu}
The bracket defined by $(\ref{Nomizu})$ satisfies the 
Jacobi identity if and only if the following two conditions hold:
\begin{enumerate}
\item  $\cyclic{X,Y,Z}\kr(X,Y,Z,V)\ =\  \sigma_T(X,Y,Z,V)$. 
\item $\cyclic{X,Y,Z}\kr(T(X,Y),Z)\ =\ 0$. 
\end{enumerate}
\end{lem}
We recognize that these two conditions are precisely
the first and second Bianchi identity for a metric connection with
parallel torsion and curvature (eqs.~(\ref{Bianchi-I}) and 
(\ref{Bianchi-II}) in Section \ref{conn-with-skew-torsion}). 
Hence, we shall call conditions
(1) and (2) the \emph{first and second Bianchi conditions}.

We now give an interpretation of the first Bianchi condition in terms
of the Clifford algebra. We embed $T\in\Lambda^3 V$ and
$\kr\in\Lambda^2 V\ox\Lambda^2 V$ in the Clifford algebra 
$\Cl(V):=\Cl(V, - \langle ,\rangle)$
by replacing the tensor product and the exterior product by the
Clifford product. Similarly, we define $T^2$ as the composition of
the endomorphism $T$ of $\Cl(V)$ with itself.
Finally, we define the (algebraic) Ricci tensor associated with
$\kr$ as $\Ric(X,Y):=\sum_{i=1}^n \kr(X,e_i,e_i,Y)$, where $e_i$ denotes
an orthonormal frame of $V$. Recall that the curvature operator
$\kr$ is a symmetric endomorphism for a connection with parallel torsion,
so these are the algebraic curvature operators we're interested in.
\begin{thm}\label{Bianchi-Clifford}
If  $\kr:\ \Lambda^2 V \ra\h\subset \Lambda^2 V $ is symmetric,
the first Bianchi condition 
is equivalent to $T^2+\kr\in\R\subset\Cl(V)$, and the second
Bianchi condition holds automatically.
\end{thm}
\begin{proof}
We express the curvature with respect to an orthonormal frame $e_1,\ldots,
e_n$ of $V$ as
\bdm
\kr\ =\ \frac{1}{4}\sum_{i,j,k,l} R_{ijkl}\, e_{ij}\ox e_{kl} \ 
\edm
A priori, $\kr$ has components of degree $0$, $2$, and $4$ in $\Cl(V)$.
The $0$-degree part vanishes, because it corresponds to summands with four
equal indices, and $R_{iiii}=0\  \forall \ 1\leq i\leq n$. The degree-$2$ part
comes from summands with exactly two equal indices that add up to the 
antisymmetric part of the algebraic Ricci tensor, which is zero for
a symmetric curvature operator (just as in the Riemannian case). Thus, 
the crucial part is to rewrite  the degree $4$-part $\kr_4$ of $\kr$
in an appropriate way. It corresponds to
summands with $i,j,k,l$ all different. To clarify  ideas, let us consider
the term $e_{1234}$. Rewrite
\bdm
\kr_4\ =\ \frac{1}{4}\sum_{i,j,k,l,\ 
 \text{all diff.}}  R_{ijkl}\, e_{ij}\ox e_{kl}
\ =\ \mu\, e_{1234}\,+ R ,
\edm
where $R$ are terms that are not proportional to  $e_{1234}$. There
is a contribution coming from all terms with $l=4$, namely,
\bea[*]
\lefteqn{\frac{1}{4}\left[R_{1234}e_{1234}+ R_{1324}e_{1324}+ R_{2134}e_{2134}
+ R_{2314}e_{2314}+ R_{3124}e_{3124}+R_{3214}e_{3214}\right]}\\
&=& \frac{1}{4}\left[R_{1234} - R_{1324}- R_{2134}
+ R_{2314} + R_{3124} - R_{3214}\right]\, e_{1234}\\
&=& \frac{1}{4}\left[2 R_{1234} 
+ R_{2314} + 2 R_{3124} - R_{3214}\right]\, e_{1234}.
\eea[*]
Similarly, one computes the contributions coming from terms
with $i,j,k=4$ and obtains
\bdm
\mu\ =\ 2(R_{1234}+ R_{3124}+ R_{2314}).
\edm
We emphasize that we did use the symmetry of the curvature
tensor in this computation. The torsion
can be written 
\bdm
T=\sum_{i<j< k} \ T_{ijk} e_{ijk},
\edm
and the standard 
identity $T^2 = -2\sigma_T+ \|T\|^2$ (\cite[Prop.3.1]{Agricola03}, 
\cite[Prop.A.1]{Agricola06})  means that the contribution in $T^2$ proportional 
to $e_{1234}$ has coefficient
\bdm
-2 \sum_{\gamma} (T_{12\gamma}T_{\gamma 34}+ T_{23\gamma} T_{\gamma 14}
+ T_{31\gamma}T_{\gamma 24}).
\edm
Finally, the first Bianchi condition states for $X=e_1,\ Y=e_2,\ Z=e_3,
V=e_4$:
\bdm
R_{1234}+ R_{3124}+R_{2314}\ =\ \sum_{\gamma} 
(T_{12\gamma}T_{\gamma 34}+ T_{23\gamma} T_{\gamma 14}
+ T_{31\gamma}T_{\gamma 24}).
\edm
Thus, we see that the term proportional to $e_{1234}$ in $T^2+\kr_4$
vanishes if ond only if the  first Bianchi condition holds.
We now prove that the second Bianchi condition is automatically
satisfied. We express $\kr$ through its action on $2$-forms, 
i.\,e.~$\kr(T(X,Y),Z) = \kr(T(X,Y)\wedge Z)$. 
Since $\kr:\ \Lambda^2 V\ra \h\subset \Lambda^2 V$ is symmetric, 
it can be written as the composition $\kr=\psi\circ\pi_\h$, where
$\pi_\h$ denotes the projection $\Lambda^2 V\ra \h$ 
(with the orthogonal complement of $\h$ inside $\so(V)$ as complementary space)
and $\psi:\ \h\ra\h$ is some $\h$-equivariant endomorphism. Hence it 
suffices to show that
\bdm
\pi_\h(\cyclic{X,Y,Z}T(X,Y)\wedge Z)\ =\ 0,
\edm
or, equivalently, that  $\cyclic{X,Y,Z}T(X,Y)\wedge Z$ is orthogonal to $\h$.
Let $\alpha$ be an element of $\h$, viewed as a skew-symmetric endomorphism 
in $\so(V)$, and $\tilde\alpha$ the corresponding $2$-form.
By definition, these satisfy
\bdm
\langle \alpha(X), Y\rangle\ =\  \tilde\alpha (X,Y)\ =\ 
\langle \tilde\alpha, X\wedge Y\rangle.
\edm
Thus, the scalar product of this cyclic sum with $\alpha$ may
be rewritten
\bdm
\cyclic{X,Y,Z} \langle \tilde \alpha, T(X,Y)\wedge Z\rangle \ = \
-  \cyclic{X,Y,Z} \langle \alpha(Z), T(X,Y)\rangle\ =\ 
\cyclic{X,Y,Z} T(X,Y,\alpha(Z)).
\edm
But the vanishing of this sum is precisely the $\h$-invariance of
$T$, which we had assumed from the very beginning.
\end{proof}
\begin{NB}
This result exists in the literature in various formulations.
It is based on an algebraic identity in the Clifford algebra
that was first observed by B.~Kostant in \cite{Kostant99} and that
is the crucial step in a formula of Parthasarathy type for
the square of the Dirac operator. The link to naturally
reductive homogeneous spaces and connections with skew torsion was 
established in the first author's work 
\cite{Agricola03} and is explained in detail in the
survey \cite{Agricola06}. As formulated here, the result was 
previously used by the
last author and N.~Schoemann in  \cite{Friedrich07} and \cite{Sch07},
but without a clear statement nor a proof. 
\end{NB}
%
%
\section{Skew holonomy systems}\label{app.SHS}
%
Our characterisation
of irreducible manifolds with vanishing $\sigma_T$ as Lie groups
(Theorem \ref{thm.sigmaT-vanishes}) relies on
the concept of skew holonomy system, that is very much inspired
by Simons' geometric approach to the proof of Berger's holonomy theorem 
\cite{Simons62}. In our exposition, we follow (mostly) the approach and
notation from \cite{OR12a}; similar results were proved independently in
\cite{Nagy13}. Partial results may already be found in \cite{Agricola&F04a}.
\begin{dfn}[{\cite{OR12a}}]\label{dfn.SHS}
Let $G\subset \SO(n)$ be a connected Lie subgroup, $V=\R^n$ the
corresponding
$G$-module, and $T\in\Lambda^3 (V)$ a $3$-form such that 
$X\haken T\in\g\subset\so(V)$ for all $X\in V$. Such a triple $(G,V,T)$ is
called a \emph{skew-torsion holonomy system}. A skew torsion holonomy system
is said to be
\emph{irreducible} if $G$ acts irreducibly on $V$,
\emph{transitive} if $G$ acts transitively on the unit sphere of $V$,
and \emph{symmetric} if $T$ is $G$-invariant.
\end{dfn}
%
%
\begin{thm}[{Skew Holonomy Theorem \cite[Thm 1.4, Thm 4.1]{OR12a}, \cite{Nagy13}}]%
\label{thm.SHS}
Let  $(G,V,T), \ T\neq 0$ be an irreducible skew-torsion holonomy system.
If it is transitive, $G=\SO(n)$. If it is not
transitive, it is symmetric, and
\begin{enumerate}
\item $V$ is an orthogonal simple Lie algebra of rank at least two
with respect to the bracket $[X,Y]=T(X,Y)$,
\item $G=\Ad(H)$, where $H$ is the connectec Lie group whose Lie algebra
is $(V, [\cdot,\cdot])$,
\item $T$ is unique, up to a scalar multiple. 
\end{enumerate}
\end{thm}
We give here one further important application of this result.
It is well-known that some manifolds carry several
connections making it naturally reductive  or, equivalently, they can be 
presented differently
as naturally reductive quotients of groups. The easiest examples are probably
the odd-dimensional spheres $S^{2n+1}= \SO(2n+2)/\SO(2n+1)= \SU(n+1)/\SU(n)$,
whose first presentation is that of a symmetric space
(the canonical connection coincides then with the Levi-Civita-connection),
the second presentation is that as a `Berger sphere' (see 
Section \ref{berger}). Furthermore, there are exceptional examples like
\bdm
S^6\ =\ G_2/\SU(3),\quad
S^7\ =\  \Spin(7)/G_2,\quad 
S^{15}=\Spin(9)/\Spin(7).
\edm
These presentations are far from accidental, they all induce interesting
$G$-structures and play a crucial role in the investigation of manifolds with
characteristic connections; $S^6$ is a nearly K\"ahler manifold and appears 
in our Classification Theorem \ref{Dn=6, classification}, see also
Remark \ref{NB-nK}. A consequence of the Skew Holonomy Theorem is
that spheres and projective spaces are basically the only manifolds 
for which this effect can
happen. 
\begin{thm}[{\cite[Thm 1.2]{OR12a}, \cite[Thm 2.1]{OR12b}}]
\label{thm.uniq-OR}
Let $(M^n,g)$ be a simply connected and irreducible Riemannian manifold
that is not isometric to a sphere, nor to a Lie group with a bi-invariant
metric or its symmetric dual. Then  $(M^n,g)$ admits at most 
\underline{one} naturally
reductive homogeneous structure.
\end{thm}
%
%
    
\vspace{2cm}
\end{document}